\documentclass[11pt,hyp,]{nyjm}
\usepackage{amsfonts, amstext, amsmath, amsthm, amscd, amssymb}
\usepackage[pdftex]{graphicx, color}

\usepackage[noadjust]{cite}
\usepackage{hyperref}
\hypersetup{nesting=true,debug=true,naturalnames=true}
\usepackage{enumitem}
\usepackage{setspace}
\usepackage{float}
\usepackage{subcaption}
\usepackage{mathrsfs}
\usepackage{tikz}
\usepackage{cleveref}
\usepackage{multicol}
\usepackage{pinlabel}
\usepackage{upgreek}

\title{Twisted knots and the perturbed Alexander invariant} 

\author{Joe Boninger}  
\address{Department of Mathematics, Boston College, Chestnut Hill, MA, 02467} 
\email{boninger@bc.edu}  
\thanks{This material is based upon work supported by the National Science Foundation under Award No.~2202704.} 



\let\<\langle
\let\>\rangle
\newcommand{\doi}[1]{doi:\,\href{http://dx.doi.org/#1}{#1}}

\let\uml\"

\newcommand{\Z}{\mathbb{Z}}
\newcommand{\N}{\mathbb{N}}
\newcommand{\R}{\mathbb{R}}

\newcommand{\C}{\mathbb{C}}

\newtheorem{thm}{Theorem}
\numberwithin{thm}{section}

\newtheorem{conj}[thm]{Conjecture}
\newtheorem{prob}[thm]{Problem}
\newtheorem{prop}[thm]{Proposition}
\newtheorem{lemma}[thm]{Lemma}
\newtheorem{cor}[thm]{Corollary}

\newtheorem*{namedtheorem}{\theoremname}
\newcommand{\theoremname}{testing}
\newenvironment{named_thm}[1]{\renewcommand{\theoremname}{#1}\begin{namedtheorem}}{\end{namedtheorem}}
\newcommand{\refthm}[1]{Theorem~\ref{thm:#1}}

\theoremstyle{definition}
\newtheorem{defn}[thm]{Definition}
\newtheorem*{nameddef}{\defname}
\newcommand{\defname}{testing}

\theoremstyle{remark}
\newtheorem{rmk}[thm]{Remark}
\newtheorem{conv}[thm]{Convention}

\begin{document}

	\begin{abstract}
		The perturbed Alexander invariant $\rho_1$, defined by Bar-Natan and van der Veen, is a powerful, easily computable polynomial knot invariant with deep connections to the Alexander and colored Jones polynomials. We study the behavior of $\rho_1$ for families of knots $\{K_t\}$ given by performing $t$ full twists on a set of coherently oriented strands in a knot $K_0 \subset S^3$. We prove that as $t \to \infty$ the coefficients of $\rho_1$ grow asymptotically linearly, and we show how to compute this growth rate for any such family. As an application we give the first theorem on the ability of $\rho_1$ to distinguish knots in infinite families, and we conjecture that $\rho_1$ obstructs knot positivity via a ``perturbed Conway invariant.'' Along the way we expand on a model of random walks on knot diagrams defined by Lin, Tian and Wang.
	\end{abstract}

	\maketitle
	\tableofcontents
	
	\section{Introduction}
	
	Let $K$ be an oriented knot in $S^3$, and $U \subset S^3$ an oriented disk intersecting $K$ in $n$ points. We will consider the set of knots $\{K_t\}_{t \geq 0}$, where $K_0 = K$ and $K_t \subset S^3$ is the result of performing $1/t$ surgery on $\partial U$ for $t > 0$. Equivalently, $K_t$ is given by inserting $t$ full twists in a regular neighborhood of $K \cap U$, as in Figure \ref{fig:full_twist}. We call such $\{K_t\}$ a {\em family of twisted knots}.
	
	Numerous authors have studied how invariants of $K_t$ behave as $t \to \infty$. Most famously, Thurston's hyperbolic Dehn surgery theorem implies that if $K \cup \partial U$ is a hyperbolic link, then the geometry of $S^3 - K_t$ asymptotically approaches the geometry of $S^3 - (K \cup \partial U)$ \cite{thu22}. Analogously, Silver and Williams proved that as $t \to \infty$, the Mahler measure of the Alexander polynomial $\Delta_{K_t}$ converges to the Mahler measure of $\Delta_{K \cup \partial U}$ \cite{sw04}. The behavior of the Jones polynomial under twisting has been studied by Champanerkar, Kofman and Yokota \cite{ck05, yok91}, and Lee has proven related results on Khovanov homology \cite{lee23}. Additionally, see \cite{roz10, wil21, aza23, che22}.
	
	Here, we examine the asymptotic effect of twisting on the {\em perturbed Alexander invariant} $\rho_1$.~$\rho_1$ is a knot invariant valued in $\Z[T,T^{-1}]$, recently defined by Bar-Natan and van der Veen \cite{bnvdv22, bnvdv21} and conjecturally related to work of Rozansky and Overbay \cite{roz96, roz98, ove13}. To study $\rho_1$ we restrict ourselves to twisted families of knots in which each point of $K \cap U$ has the same sign---in this case we say a regular neighborhood of $K \cap U$, or the family $\{K_t\}$, is {\em coherently oriented}. We prove the coefficients of the polynomials $\{\rho_1(K_t)\}$ grow asymptotically linearly and that this growth rate converges to a rational function.
	\begin{thm}
		\label{thm:main}
		For a family of knots $\mathcal{K} = \{K_t\}$ as above, twisted along $n$ coherently oriented parallel strands, define
		$$
		d_t(\mathcal{K}) = T^{tn(n - 1)} \Big( T^{n(n - 1)}\rho_1(K_{t + 1}) - \rho_1(K_t) \Big).
		$$
		Then as $t \to \infty$, $d_t$ converges to a non-polynomial rational function in $T$. The limit, considered as a rational function, satisfies
		$$
		(\lim_{t \to \infty} d_t(\mathcal{K}))|_{T = 1} = \pm \frac{n - 1}{2n}.
		$$
	\end{thm}

	\begin{figure}[t]
		\includegraphics[height=1.3cm]{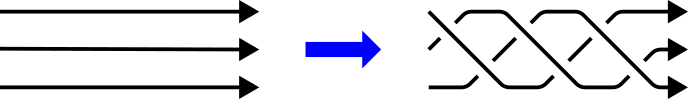}
		\caption{Performing a full twist on three coherently oriented strands}
		\label{fig:full_twist}
	\end{figure}

	An analogous result holds as $t \to -\infty$, but we focus on the positive case for brevity. We call the limit $\lim_{t \to \infty} d_t(\mathcal{K})$ the {\em asymptotic growth rate} of $\rho_1$ for the family $\mathcal{K}$; for our precise notion of convergence see Definition \ref{def:stab} below.
	
	The family of $(2,q)$-torus knots $\{{\bf T}(2,2t + 1)\}_{t \geq 0}$ provides a concrete example of this convergence. Their $\rho_1$ invariants for small $t$ are:
	{\small
		\begin{align*}
			\rho_1(\text{unknot}) &= 0 \\
			\rho_1(\textbf{T}(2,3)) &= T^{-2}(-1 + 2T - 2T^2 + 2T^3 - T^4) \\
			\rho_1(\textbf{T}(2,5)) &= T^{-4}(-2 + 4T - 5T^2 + 6T^3 - 6T^4 + 6T^5 - 5T^6 + 4T^7 - 2T^8) \\
			\rho_1(\textbf{T}(2,7)) &= T^{-6}(-3 + 6T - 8 T^2 + 10 T^3 - 11 T^4 + 12 T^5 - 12T^6 + 12T^7 - 11T^8 + 10T^9 \\
				& \ \ \ \ \ \ \ \ \ \ - 8 T^{10} + 6 T^{11} - 3 T^{12})
		\end{align*}
	}
	
	From these data points, the asymptotic growth rate of $\rho_1$ appears to be
	$$
		-1 + 2T - 3T^2 + 4T^3 - \cdots =-\frac{1}{(1 + T)^2}.
	$$
	Indeed, we prove:
	\begin{thm}
		\label{thm:example}
		The asymptotic growth rate of $\rho_1$ for the family of $(2,q)$-torus knots is given by
		$$
		\lim_{t \to \infty} d_t(\{{\bf T}(2, 2t + 1)\}) = -\frac{1}{(1 + T)^2}.
		$$ 
	\end{thm}
	Moreover, we show how the asymptotic growth rate can be easily computed for any coherently oriented, twisted family of knots.
	
	Since the asymptotic growth rate is always non-zero by Theorem \ref{thm:main}, we also obtain:
	\begin{cor}
		\label{cor:distinguishing}
		Let $\{K_t\}_{t \geq 0}$ be a coherently oriented family of twisted knots. Then $\rho_1$ distinguishes infinitely many knots in this family. In fact, for some $t_0 \in \N$, the polynomials $\{\rho_1(K_t)\}_{t \geq t_0}$ are all distinct.
	\end{cor}
	Corollary \ref{cor:distinguishing} is a modest result since its conclusion also applies to the Alexander polynomial, but it is the first result we know of on the ability of $\rho_1$ to distinguish knots in infinite families.
	
	Our primary motivation in writing this paper was understanding the striking pattern of Theorem \ref{thm:example}, but there are many additional reasons why the invariant $\rho_1$ warrants study. Bar-Natan and van der Veen show $\rho_1$ is computable in polynomial time, unlike the Jones polynomial or any knot homology theory as of this writing. In spite of this the pair of invariants $(\Delta, \rho_1)$, $\Delta$ the Alexander polynomial, is superior to the HOMFLY-PT polynomial and Khovanov homology combined at distinguishing knots with up to 11 crossings \cite{bnvdv22}\cite[Theorem 50]{bnvdv21}. The polynomial $\rho_1$ also has deep connections to both the Alexander and colored Jones polynomials---as its name suggests, the perturbed Alexander invariant $\rho_1$ is one of an infinite family of polynomial knot invariants resulting from perturbing a quantum group associated to $\Delta$. The full family of perturbed Alexander invariants determines all colored Jones polynomials for any knot \cite{bnvdv21}. Thus, $\rho_1$ is a promising candidate for building bridges between classical and quantum topology.
	
	\subsection{Random walks, and full twists on braids}
	
	In contrast to the Alexander polynomial, very little machinery exists to study the $\rho_1$ invariant. Thus we develop novel tools and techniques, and our proofs of Theorems \ref{thm:main} and \ref{thm:example} rely on two sets of secondary results which may be of independent interest. First, we build on a model of random walks on tangle diagrams introduced by Lin, Tian and Wang \cite{ltw98} following a remark of Jones \cite{jon87}. Lin, Tian and Wang's framework associates a Markov chain to any tangle diagram, and the Alexander polynomial and $\rho_1$ invariant can both be defined as sums over certain expected values in this Markov chain \cite{lw01, bnvdv22}.
	
	Broadly speaking, we prove Theorem \ref{thm:main} by showing that the relevant expected values in any tangle Markov chain stabilize under twisting. To do this, we define a contraction operation on such Markov chains in \Cref{sec:cmc} below. In Propositions \ref{thm:cmc} and \ref{thm:cmc_det}, we show that our contraction operation is natural in the sense that it preserves many properties of the original chain. Contracting allows us to consider subdiagrams of a given knot diagram---for example, a twist region---and draw global conclusions about its $\rho_1$ invariant. We believe this tool will be helpful in further study of $\rho_1$ and the higher order perturbed Alexander invariants defined in \cite{bnvdv21}.
	
	In our second group of supplementary results, in \Cref{sec:burau}, we compute the (unreduced) Burau representations of powers of full twists in the $n$-strand braid group $B_n$, for all $n$. The Burau representation is related to the random walks model in the following way: for any braid $\beta \in B_n$, if $\psi(\beta) \in \text{GL}_n(\Z[T, T^{-1}])$ denotes its Burau representation, then the probability of a random walk entering $\beta$ at the $i$th strand and exiting at the $j$th strand is equal to $\psi(\beta)_{ij}$ \cite{ltw98}. In the process of proving the stabilization results above, we prove that the sequence of Burau representations $\{\psi(\Omega_n^k)\}_k$ stabilizes as $k \to \infty$ for all $n$, where $\Omega_n \in B_n$ denotes a positive full twist. In fact, we give the limit
	$$
	\psi(\Omega_n^\infty) = \lim_{k \to \infty} \psi(\Omega_n^k)
	$$
	explicitly as a matrix of rational functions. As we mention in \Cref{rem:roz} below, it is interesting to compare this result to Rozansky's observation that the Jones-Wenzl idempotent of the Temperley-Lieb algebra can be defined as the limit of Temperley-Lieb representations of the sequence $\{\Omega^k_n\}_k$ \cite{roz14}. In this context, our matrix $\psi(\Omega_n^\infty)$ can viewed as an analogue of the Jones-Wenzl idempotent for the Burau representation.

	\subsection{A perturbed Conway invariant}
	
	In this section, which is independent from the rest of the paper, we discuss a natural question arising from our work. Here only, we assume two properties of the invariant $\rho_1$ for any knot $K \subset S^3$:
	\begin{itemize}
		\item $\rho_1(K)$ is divisible by $(1 - T)^2$.
		\item $\rho_1(K)$ is symmetric, i.e.~satisfies $\rho_1(K)(T) = \rho_1(K)(T^{-1})$.
	\end{itemize}
	Both properties have been conjectured by Bar-Natan and van der Veen  \cite{bnvdv21} and have been verified for knots with up to 11 crossings. Assuming these, we define a symmetric Laurent polynomial by
	$$
	\rho_1^\text{red}(K) = \frac{T}{(1 - T)^2} \cdot \rho_1(K)
	$$
	for any knot $K$.
	
	 As the previous section discusses, our proofs rely on a model of random walks on knot diagrams. The phrase ``random walk'' is used loosely, however, since in most cases this framework involves ``probabilities'' that do not lie in the unit interval. An oriented knot diagram $D \subset S^2$ gives an honest random walk model precisely when it is {\em positive}, i.e.~every crossing has a positive sign. A knot admitting such a diagram is also called {\em positive}, and it is thus natural to expect the perturbed Alexander invariant to obstruct positivity.
	
	For any knot $K$, the {\em Conway polynomial} $\nabla_K(z) \in \Z[z]$ is the unique polynomial satisfying
	$$
	\nabla_K(x - x^{-1}) = \Delta_K(x^2),
	$$
	where $\Delta_K$ is the symmetrized Alexander polynomial. Cromwell proved that if $K$ is positive then the coefficients of $\nabla_K$ are non-negative \cite[Corollary 2.1]{cro89}. This motivates the following definition and conjecture.
	
	\begin{defn}
		For any knot $K \subset S^3$, the {\em perturbed Conway invariant} $\delta_1 \in \Z[z]$ is the unique polynomial satisfying
		$$
		\delta_1(K)(x - x^{-1}) = \rho^\text{red}_1(K)(x^2).
		$$
	\end{defn}
	The existence and uniqueness of $\delta_1$ are easy to check using the assumptions above.
	
	\begin{conj}
		\label{conj:positivity}
		If $K \subset S^3$ is a positive knot, then the coefficients of $\delta_1(K)$ are all negative or zero.
	\end{conj}

	We note the sign difference with Cromwell's result but also remark that unlike the Alexander polynomial, the perturbed Alexander invariant is sensitive to mirroring. We could equivalently have conjectured that if $K$ admits a diagram with all negative crossings then the coefficients of the perturbed Conway invariant are zero or positive.

	We have verified \Cref{conj:positivity} directly for all knots with up to $10$ crossings. Additionally, if Conjecture \ref{conj:positivity} is true, then $\rho_1$ obstructs positivity in cases where the Alexander polynomial fails---the knots $9_{29}$ and $10_{19}$ are two such examples. We also suspect an analog of Conjecture \ref{conj:positivity} may hold for the higher order perturbed Alexander invariants.
	
	\subsection{Further Discussion}
	
	Little is known about the $\rho_1$ invariant, as the previous section suggests. In addition to Conjecture \ref{conj:positivity}, we mention two other problems relevant to our work.
	
	\begin{prob}
		Describe the asympotic behavior of $\rho_1$ for families of twisted knots which are not coherently oriented.
	\end{prob}
	
	\begin{prob}
		\label{prob:tk}
		Find a closed formula for the $\rho_1$ invariant of torus knots.
	\end{prob}

	Conjecturally, $\rho_1$ is a specialization of the {\em 2-loop polynomial}---see \cite[Conjecture 7.11]{bec24}. March\'e and Ohtsuki independently computed the 2-loop polynomial for torus knots \cite{mar04, oht04}, which provides a solution to Problem \ref{prob:tk} if \cite[Conjecture 7.11]{bec24} is true. Nonetheless, it would be interesting to calculate the $\rho_1$ invariant of torus knots directly from the definitions given in \cite{bnvdv21} or \cite{bnvdv22}, which may also affirm \cite[Conjecture 7.11]{bec24} for this case.
	
	\subsection{Outline}
	
	Section \ref{sec:background} reviews the $\rho_1$ invariant and the random walks model of Lin, Tian and Wang. Section \ref{sec:cmc} then defines our contraction operation on tangle Markov chains, and Section \ref{sec:burau} computes the Burau representations of powers of full twists in the braid group. In Section \ref{sec:alex}, as a necessary prerequisite to our main results, we prove a stabilization result for the Alexander polynomial. We then prove Theorems \ref{thm:main} and \ref{thm:example} in Section \ref{sec:rho}, and show how the asymptotic growth rate of $\rho_1$ can be computed for any coherently oriented family of twisted knots. Finally, Section \ref{sec:lemma} contains the proof of a technical lemma used in Section \ref{sec:cmc}.
	
	\subsection{Acknowledgements}
	
	The author would like to thank Ilya Kofman, Josh Greene and an anonymous referee for suggesting improvements to earlier drafts of this paper. The author is also grateful to Dror Bar-Natan, Roland van der Veen, Hans Boden, Abhijit Champanerkar and Dan Silver for encouraging and insightful conversations, and to seminar attendees at MIT and the CUNY Graduate Center for listening to talks on this project while it was a work in progress and giving feedback.
	
	\section{Background and conventions}
	\label{sec:background}
	
	\subsection{Upright diagrams and the $\rho_1$ invariant}
	\label{sec:ud}
	
	A {\em tangle} is a proper embedding $L \hookrightarrow \Sigma \times I$, considered up to ambient isotopy, where $L$ is an oriented 1-manifold and $\Sigma \subset \R^2$ a compact planar surface. A {\em tangle diagram} $D \subset \Sigma$ is the oriented image of a generic projection $L \hookrightarrow \Sigma \times I \to \Sigma$, with over-under information added at crossing points. Additionally, we say the diagram $D \subset \Sigma$ is {\em upright} if, near each crossing point, the two intersecting regions of $D$ are oriented upward in $\R^2$ (i.e.~their $y$-velocity is positive). We denote the set of crossing points in a tangle diagram $D \subset \Sigma$ by $\mathcal{C}$, and we call the connected components of $D - \mathcal{C}$ {\em strands}. We say a strand which intersects $\partial \Sigma$ is {\em incoming} if its orientation points into $\Sigma$, and {\em outgoing} otherwise.
	
	To define the $\rho_1$ invariant for a knot $K \subset S^3$ we represent $K$ by an upright {\em long knot diagram} $D$, an immersed arc in the unit square $I^2$. Let $\mathcal{C}$ be the set of crossing points of $D$, $n = |\mathcal{C}|$, and label the strands of $D$ consecutively from $k = 1$ to $2n + 1$, where $1$ is the incoming strand and $2n + 1$ the outgoing one. Let $\varphi_k$ denote the {\em turning number} of the strand $s_k$, the number of full counterclockwise turns $s_k$ makes in the plane minus the number of clockwise turns, and let $\varphi(D)$ be the total turning number of $D$. These are well-defined integers since $D$ is upright.
	
	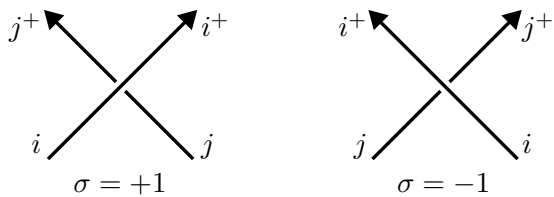
\begin{figure}[t]
		\labellist
		\small\hair 2pt
		\pinlabel {$\sigma = +1$} at 63 -20
		\pinlabel {$\sigma = -1$} at 331 -20
		\pinlabel $j$ [tl] at 127 25
		\pinlabel $i^+$ [tl] at 127 125
		\pinlabel $i$ [tr] at 1 25
		\pinlabel $j^+$ [tr] at 1 125
		\pinlabel $i$ [tl] at 393 25
		\pinlabel $j^+$ [tl] at 393 125
		\pinlabel $j$ [tr] at 271 25
		\pinlabel $i^+$ [tr] at 271 125
		\endlabellist
		
		\includegraphics[height=2cm]{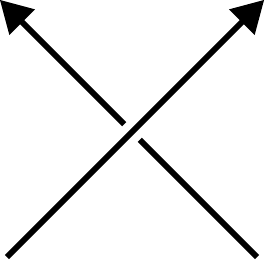}
		\hspace{2cm}
		\includegraphics[height=2cm]{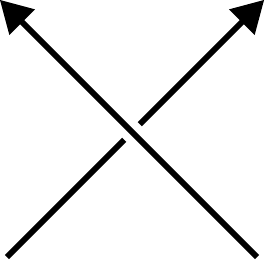}
		\vspace{.25cm}
		\caption{Crossing conventions}
		\label{fig:cc}
	\end{figure}

	\begin{figure}[t]
		\labellist
		\small\hair 2pt
		\pinlabel $1$ [l] at 15 20
		\pinlabel $2$ at 100 70
		\pinlabel $3$ [l] at 15 110
		\endlabellist
		
		\includegraphics[height=3cm]{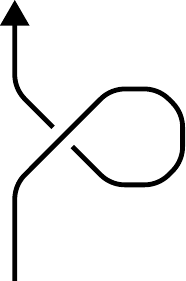}
		\caption{An upright long knot diagram of the unknot}
		\label{fig:kink}
	\end{figure}

	We represent each crossing $c \in \mathcal{C}$ by a triple $(\sigma, i, j)$ where $\sigma$ is the sign of $c$, $i$ is the label of its incoming over-strand, and $j$ is the label of its incoming under-strand---see Figure \ref{fig:cc}. Additionally, we let $i^+$ and $j^+$ be the respective labels of the outoing over- and under-strands of $c$. In this case we have $i^+ = i + 1$ and $j^+ = j + 1$, but it will be convenient later to allow for other possible labelings. For each $c = (\sigma, i, j)$, define the matrix $A_c = (a_{k\ell}) \in \text{M}_{2n + 1}(\Z[T,T^{-1}])$ by
	\begin{align*}
		a_{i,i^+} &= T^\sigma & a_{i, j^+} &= 1 - T^\sigma \\ 
		a_{j, j^+} &= 1 & a_{k\ell} &= 0 \text{ otherwise}.
	\end{align*}
	We then set $A = \sum_{c \in \mathcal{C}} A_c$. For example, the labeled diagram shown in Figure \ref{fig:kink} has the matrix $A$ and $\varphi$ values shown below.
	\begin{align*}
		A &= \begin{bmatrix}
			0 & T & 1 - T \\
			0 & 0 & 1 \\
			0 & 0 & 0
		\end{bmatrix} &
		\varphi_1 = \varphi_3 = 0, \ \varphi_2 = -1
	\end{align*}

	The matrix $I - A$, where $I$ is the identity matrix of appropriate dimension, can be obtained from Fox calculus using the Wirtinger presentation. If $\Delta_K$ is the symmetrized Alexander polynomial of $K$ and $w(D)$ the writhe of $D$, then $I - A$ satisfies
	\begin{equation}
		\label{eq:alex}
		\Delta_K(T) = T^{(-\varphi(D) - w(D))/2}\det(I - A).
	\end{equation}
	In particular $I - A$ is invertible in the ring of rational functions $\Z(T)$, and we denote its inverse by
	$$
	G = (g_{k\ell}) = (I - A)^{-1} \in M_{2n + 1}(\Z(T)).
	$$
	Finally, if a crossing $c \in \mathcal{C}$ is represented by the triple $(\sigma, i, j)$ as above, we define $R_1(c) \in \Z(T)$ by
	\begin{equation}
		\label{eq:r}
		R_1(c) = \sigma(g_{ji}(g_{j^+, j} + g_{j, j^+} - g_{ij}) - g_{ii}(g_{j,j^+} - 1) - 1/2).
	\end{equation}
	
	\begin{defn}{\protect\cite{bnvdv22}}
		\label{def:rho}
		The {\em perturbed Alexander invariant} $\rho_1 \in \Z[T,T^{-1}]$ is given by the formula
		$$
		\rho_1(K) = \Delta_K^2 \Big( \sum_{c \in C} R_1(c) - \sum_{k = 1}^{2n + 1} \varphi_k (g_{kk} - 1/2) \Big),
		$$
		where $\Delta_K$ is the symmetrized Alexander polynomial of $K$.
	\end{defn}

	Since $R_1$ is quadratic in $G$ and $\det(G)$ is a normalization of $1/\Delta_K$, the factor $\Delta_K^2$ in Definition \ref{def:rho} ensures $\rho_1$ is a genuine Laurent polynomial. It is shown in \cite{bnvdv22} that $\rho_1$ is a knot invariant, and for an explanation of the theory underlying $\rho_1$ we refer the reader to \cite{bnvdv21}. In the next subsection we provide some intuition for the matrices $A$ and $G$.
	
	\begin{rmk}
		As this section shows, $\rho_1$ can be most accurately thought of as an invariant of {\em long knots} in the unit cube, which yields a knot invariant by gluing the ends of a given long knot together in $\R^3$. Our definition of $\rho_1$ also extends naturally to $(1,1)$-tangles, i.e.~tangles with one arc conomponent, at least when the matrix $I - A$ is invertible. This latter construction does not give a well-defined link invariant, however, since there is no canonical way to produce a $(1,1)$-tangle from a multi-component link. Defining $\rho_1$ for links remains an open problem.
	\end{rmk}
	
	\subsection{Random walks on tangle diagrams}
	
	The matrices $A$ and $G$ of Section \ref{sec:ud} fit into a model of random walks on tangle diagrams due to Jones \cite{jon87} and Lin, Tian and Wang \cite{ltw98}. In this model a particle (or car \cite{bnvdv22}, or bowling ball \cite{jon87}) placed anywhere on a diagram moves along it in the direction of orientation. When the particle passes over a crossing with sign $\sigma$, it continues on the upper strand with ``probability'' $T^\sigma$ and ``jumps down'' to the lower strand of the crossing with ``probability'' $1 - T^\sigma$. From there it follows the orientation as before. The values $T$ and $T^{-1}$ do not both lie in $[0,1]$ for $T \neq 1$, so these are not probabilities in a strict sense. Nevertheless, this framework has practical applications.
	
	Lin, Tian and Wang's random walk model is an example of a {\em Markov chain} $M = (\mathcal{S}, a)$, which for our purposes consists of a finite set of {\em states} $\mathcal{S} = \mathcal{S}(M)$ and a {\em transition function}
	$$
	a = a_M : \mathcal{S} \times \mathcal{S} \to R,
	$$
	where $R$ is a commutative ring with unity. We think of $a(s,t)$, $s,t \in \mathcal{S}$, as the probability or {\em weight} that the next state of a particle will be $t$ if its current state is $s$.
	
	If the states in $\mathcal{S}$ are indexed $1, \dots, k$, we define a {\em transition matrix} $A = A_M = (a_{ij}) \in \text{M}_k(R)$ by
	$$
	a_{ij} = a(s_i, s_j), \ \ \ \ \ s_i, s_j \in \mathcal{S}.
	$$
	In our random walks framework, the states $\mathcal{S}(D)$ of a tangle diagram $D$ are its strands, the ring $R$ is $\Z[T,T^{-1}]$, and the transition matrix is precisely the matrix $A$ of Section \ref{sec:ud}. Henceforth, we set $R = \Z[T,T^{-1}]$.
	
	For a Markov chain $M = (\mathcal{S}, a)$ and $s,t \in \mathcal{S}$, a {\em walk} $w$ from $s$ to $t$ is a finite sequence of states $s = s_1s_2 \cdots s_k = t$, such that $a(s_i, s_{i + 1}) \neq 0$ for $i = 1, \dots, k - 1$. The {\em length} of the walk is $\ell(w) = k$, and the {\em weight of $w$} is
	$$
	a(w) = a_M(w) = \prod_{i = 1}^{k - 1} a(s_i,s_{i+1}).
	$$
	Note that $a(s_1s_2) = a(s_1,s_2)$. By convention if $\ell(w) = 1$, i.e.~if $w = s = t$, then $a(w) = 1$. The weight of $w$ can be thought of as the probability of $w$ occurring among all length $k$ walks with initial state $s$.
	
	We let $\mathcal{W}_{s,t} = \mathcal{W}_{s,t}(M)$ denote the set of all walks from $s$ to $t$, and we define the {\em Green's function} $g(s,t)$ by
	\begin{equation}
		\label{eq:prob_def}
		g(s,t) = g_M(s,t)= \sum_{w \in \mathcal{W}_{s,t}} a(w)
	\end{equation}
	whenever the righthand sum converges. Similarly, if $\mathcal{S'} \subset \mathcal{S}$, let $\mathcal{W}_{s,t}|_\mathcal{S'} = \mathcal{W}_{s,t}(M)|_\mathcal{S'} $ be the set of walks from $s$ to $t$ which contain only states in $\mathcal{S'}$. Define
	$$
	g|_{\mathcal{S}'}(s,t) =  g_M|_{\mathcal{S'}}(s,t) =  \sum_{w \in \mathcal{W}_{s,t}|_\mathcal{S'}} a(w).
	$$
	We think of $g(s,t)$ as the expected number of times a walk $w$ will contain the state $t$, given that its initial state is $s$.
	
	Lin, Tian and Wang prove the following:
	\begin{lemma}[\protect{\cite[Theorem A, Proposition 2.1]{ltw98}}]
		\label{lem:ltw}
		Let $D \subset I^2$ be a tangle diagram of a $1$-manifold with no closed components. Then there exists an open neighborhood $U$ of $1 \in \C$ such that for all $T \in U$ and all $s,t \in \mathcal{S}(D)$, the sum (\ref{eq:prob_def}) converges absolutely to a rational function in $\Z(T)$. It follows that the function $g$ is well defined in this case.
	\end{lemma}
	If $D$ is a tangle diagram with no closed link components, then certainly (\ref{eq:prob_def}) converges when $T = 1$. At $T = 1$ any walk containing a jump has weight zero, and in a diagram with no closed link components there is at most one jump-less walk between any pair of states. (This is the walk given just by following the diagram.) Lin, Tian and Wang essentially show the series (\ref{eq:prob_def}) centered at $T = 1$ is similar enough to a geometric series to have a positive radius of convergence.
	
	For any Markov chain $M$ with indexed states $\mathcal{S} = \{s_1, \dots, s_k\}$ and transition matrix $A$, if $I - A$ is invertible we define
	$$
	G = G_M = (g_{ij}) := (I - A)^{-1}
	$$
	as in Section \ref{sec:ud}. Following Bar-Natan and van der Veen \cite{bnvdv22}, we call the matrix $G$ a {\em Green's matrix} for the Markov chain. The next lemma explains this name; cf.~\cite[Proposition 5]{bnvdv22}.
	
	\begin{lemma}
		\label{lem:path_conv}
		Let $M$ be a Markov chain with transition function $a : \mathcal{S}(M) \to \Z(T)$. If an evaluation $T \in \C$ is chosen so that $g(s_i,s_j)$ converges at for all $s_i,s_j \in \mathcal{S}(M)$, then $I - A$ is invertible and
		$$
		g(s_i,s_j) = g_{ij}
		$$
		for all $i$ and $j$. In particular the above holds for all tangle diagrams with no closed components, for all $T$ in an open neighborhood of $1$.
	\end{lemma}

	\begin{proof}
		An inductive argument shows that for all $s_i, s_j \in \mathcal{S}(M)$ and all $k \geq 1$,
		$$
		\sum_{\{w \in \mathcal{W}_{s_i,s_j} \mid \ell(w) = k\}} a(w) = (A^{k - 1})_{ij}.
		$$
		Thus
		$$
		g(s_i,s_j) = \sum_{k = 1}^\infty (A^{k - 1})_{ij} = \big(\sum_{k = 1}^\infty A^{k - 1} \big)_{ij}.
		$$
		Since $g(s_i,s_j)$ is well defined for all $i$ and $j$, the geometric series on the right converges at the chosen value of $T$ and we have
		$$
		g(s_i,s_j) = \big(\sum_{k = 1}^\infty A^{k - 1} \big)_{ij} = (I - A)^{-1}_{ij}
		$$
		as desired.
	\end{proof}
	
	Moving forward we use $g_{ij}$ and $g(s_i,s_j)$ interchangeably, and when dealing with the latter we implicitly choose $T$ so that (\ref{eq:prob_def}) converges absolutely.
	
	\subsection{Simple cycles and the Alexander polynomial}
	\label{sec:hikes}
	
	Let $M = (\mathcal{S}, a)$ be a Markov chain with transition map $a$ valued in $\Z(T)$. Given a finite sequence of states $c = s_1s_2 \cdots s_k \in \mathcal{S}^k$, let $[c]$ denote the equivalence class of $c$ considered up to cyclic permutation of the sequence. Following \cite{lw01, foze99}, define
	$$
	a_\text{circ}([c]) = a(s_1,s_2)a(s_2,s_3) \cdots a(s_{k - 1}, s_k)a(s_k,s_1).
	$$
	Then $a_\text{circ}$ is well defined, and we call $[c]$ a {\em cycle} of $M$ if $c$ is non-empty and $a_\text{circ}([c]) \neq 0$. A cycle is {\em simple} if it does not contain any repeated states.
	
	A {\em multicycle} is a finite set of cycles, and a multicycle $q$ is {\em simple} if each of its cycles is simple and if no state appears in more than one cycle of $q$. For any multicycle $q = \{c_1, \dots, c_k\}$, we define $|q| = k$ and
	$$
	a_\text{circ}(q) = \prod_{i = 1}^k a_\text{circ}(c_i).
	$$
	By convention, the empty multicycle $\varnothing$ is a simple multicycle satsifying $|\varnothing| = 0$ and
	$$
	a_\text{circ}(\varnothing) = 1.
	$$
	
	Let $A$ be a transition matrix for $M$, $G$ a Green's matrix, and let $\mathcal{Q} = \mathcal{Q}(M)$ denote the (finite) set of all simple multicycles of $M$. Then we have the following classical identity:
	\begin{equation}
		\label{eq:cf}
		\sum_{q \in \mathcal{Q}} (-1)^{|q|} a_\text{circ}(q) =  \det(I - A) = \det(G)^{-1}.
	\end{equation}
	This equation follows from seminal work of Cartier and Foata \cite{cafo69} and a proof can be found in \cite[Section 2.3]{giro17}. When $M$ is a markov chain induced by a long knot diagram, the above equation also gives the Alexander polynomial of the knot by (\ref{eq:alex}).
	
	\begin{rmk}
		There is a dual identity to (\ref{eq:cf}) which gives $\det(G)$, rather than $\det(G)^{-1}$, in terms of simple cycles. This alternate identity is used in \cite[Theorem 4.3]{lw01} to study the Alexander polynomial, although its statement in that paper is incorrect. The correct version can be found, for example, in \cite[Section 2.3]{giro17}.
	\end{rmk}
		
	\section{Contracting tangle Markov chains}
	\label{sec:cmc}
	
	In this section we define a contraction operation on tangle Markov chains which will be vital in the proofs of our main theorems. Let $\Sigma \subset \R^2$ be a planar surface and $D \subset \Sigma$ a tangle diagram with no closed link components. Let $U \subset \text{int}(\Sigma)$ be a compact subsurface of $\Sigma$ such that its boundary $\partial U$ is tranverse to $D$ and avoids crossing points. Then $U$ defines two subdiagrams of $D$: $D \cap U$ and $D \cap (\Sigma - U)$. Abusing notation, we denote the former by $U$ and the latter by $D - U$, and we identify the states $\mathcal{S}(U)$ and $\mathcal{S}(D - U)$ with subsets of $\mathcal{S}(D)$ via the inclusions $U, D - U \hookrightarrow D$. The intersection $\mathcal{S}(U) \cap \mathcal{S}(D - U)$ consists precisely of the incoming and outgoing strands of $U$; we label the former by $\mathcal{I}$ and the latter by $\mathcal{O}$.
	
	\begin{defn}
		\label{def:cmc}
		With notation as above, we define a Markov chain $D/U$ as follows:
		\begin{itemize}
			\item The states of $D/U$ are those of $D - U$: 
			$$
				\mathcal{S}(D/U) = \mathcal{S}(D - U).
			$$
			\item The transition function of $D/U$ is defined by
			$$
			a_{D/U}(s,t) = \begin{cases}
				g_D|_{\mathcal{S}(U)}(s,t) & s \in \mathcal{I}, t \in \mathcal{O} \\
				a_D(s,t) & \text{otherwise}
			\end{cases}.
			$$
			
		\end{itemize}
		We call $D/U$ the {\em contraction of $D$ by $U$}.
	\end{defn}

	The Markov chain $D/U$ can be drawn diagrammatically by replacing $U \subset D$ with a junction or vertex, similar to Figure \ref{fig:d_inf} in Section \ref{sec:alex}---this vertex sends particles entering on strands in $\mathcal{I}$ to strands in $\mathcal{O}$ with weights determined by $U$. By definition, the transition function takes values in $\Z(T)$. 
	
	The following two propositions show contracting is a natural operation---the first shows that it preserves expected values.
	\begin{prop}
		\label{thm:cmc}
		The Green's function of $D/U$,
		$$
		g_{D/U} : \mathcal{S}(D/U) \times \mathcal{S}(D/U) \to \Z(T),
		$$
		 is well-defined for all $T \in \C$ in an open neighborhood of $1$. For all $s,t \in \mathcal{S}(D/U) \subset \mathcal{S}(D)$,
		 $$
		 	g_{D/U}(s,t) = g_D(s,t).
		 $$
	\end{prop}

	\begin{proof}
		The proof of \cite[Theorem A]{ltw98} tells us the Green's function $g_{D/U}$ is well defined provided $D/U$ satisfies the following condition: when $T = 1$, the value $a_{circ}(q)$ of any cycle $q$ in $D/U$ must be $0$. This condition certainly holds for $D$---since $D$ has no closed components, any cycle contains at least one ``jump'' at a crossing. When $T = 1$, as the previous section discusses, no jumps occur. Given this, it is not difficult to check that the condition holds for $D/U$ as well, and we leave the details of this to the reader.
		
		We now address the second statement of the Proposition. Fix $s,t \in \mathcal{S}(D/U) \subset \mathcal{S}(D)$, and let $\widehat{\mathcal{W}}_{s,t}(D/U)$ denote the set of all finite sequences of states in $\mathcal{S}(D/U)$ with initial state $s$ and terminal state $t$. Then $\mathcal{W}_{s,t}(D/U) \subset \widehat{\mathcal{W}}_{s,t}(D/U)$. Furthermore, for any $\mathcal{W}'$ satisfying $\mathcal{W}_{s,t}(D/U) \subset \mathcal{W}' \subset  \widehat{\mathcal{W}}_{s,t}(D/U)$, we have
		$$
		g_{D/U}(s,t) = \sum_{w \in \mathcal{W}_{s,t}(D/U)} a_{D/U}(w) = \sum_{w \in \mathcal{W}'} a_{D/U}(w)
		$$
		since all sequences $w \in \widehat{\mathcal{W}}_{s,t}(D/U) - \mathcal{W}_{s,t}(D/U)$ have $a_{D/U}(w) = 0$ by definition. We define a map
		$$
		\pi : \mathcal{W}_{s,t}(D) \to \widehat{\mathcal{W}}_{s,t}(D/U)
		$$
		as follows: given $z \in \mathcal{W}_{s,t}(D)$, let $\pi(z) \in \widehat{\mathcal{W}}_{s,t}(D/U)$ be the sequence defined by removing all states of $\mathcal{S}(D) - \mathcal{S}(D/U)$ from $z$.
		
		We claim $\mathcal{W}_{s,t}(D/U) \subset \text{Im}(\pi)$. To this end, let $w = s_1s_2 \cdots s_\ell \in \mathcal{W}_{s,t}(D/U)$. If $w$ contains no adjacent $s_i, s_{i + 1}$ with $s_i \in \mathcal{I}$ and $s_{i + 1} \in \mathcal{O}$, then
		$$
		0 \neq a_{D/U}(w) = \prod_{i = 1}^{\ell - 1} a_{D/U}(s_i, s_{i + 1}) = \prod_{i = 1}^{\ell - 1} a_D(s_i, s_{i + 1}) = a_D(w),
		$$
		so $w \in \mathcal{W}_{s,t}(D)$ and $\pi(w) = w$. If $w$ contains one such pair---say, $s_1 \in \mathcal{I}$ and $s_2 \in \mathcal{O}$---then $a_{D/U}(w) \neq 0$ implies
		$$
		0 \neq a_{D/U}(s_1,s_2) = g_D|_{\mathcal{S}(D_U)}(s_1,s_2).
		$$
		Therefore there exists a walk $z' \in \mathcal{W}_{s_1,s_2}(D)|_{\mathcal{S}(D_U)}$, and we define $z = z's_3s_4 \cdots s_m$. It is straightforward to check that $a_D(z) \neq 0$ and $\pi(z) = w$, so $w \in \text{Im}(\pi)$ in this case as well. For general $w$, we construct a walk $z \in \mathcal{W}_{s,t}(D)$ with $\pi(z) = w$ by replacing each pair $s_i \in \mathcal{I}$, $s_{i + 1} \in \mathcal{O}$ with a walk $z_i' \in \mathcal{W}_{s_i,s_{i + 1}}(D)|_{\mathcal{S}(D_U)}$. This proves the claim.
		
		Next we claim that for all $w \in \text{Im}(\pi)$, if $\pi^{-1}(w)$ denotes the preimage of $w$ in $\mathcal{W}_{s,t}(D)$, then
		$$
		\sum_{z \in \pi^{-1}(w)} a_D(z) = a_{D/U}(w).
		$$
		Let $\tilde{\pi}$ denote the obvious extension of $\pi$ to all walks on $D$: for any walk $z$ on $D$, let $\tilde{\pi}(z)$ be the sequence of states formed by removing all states from $z$ which are not in $\mathcal{S}(D/U)$. Fix $s_1, s_2 \in \mathcal{S}(D/U)$ with $s_1s_2 \in \text{Im}(\tilde{\pi})$, and suppose $s_1 \notin \mathcal{I}$ or $s_2 \notin \mathcal{O}$. In this case there is no walk from $s_1$ to $s_2$ in $D$ with all intermediate states in $\mathcal{S}(D) - \mathcal{S}(D/U)$, so $\tilde{\pi}^{-1}(s_1s_2) = \{s_1s_2\}$. Therefore
		$$
		\sum_{z \in \tilde{\pi}^{-1}(s_1s_2)} a_D(z) = a_D(s_1s_2) = a_{D/U}(s_1s_2).
		$$
		On the other hand, if $s_1 \in \mathcal{I}$ and $s_2 \in \mathcal{O}$, then $\tilde{\pi}^{-1}(s_1s_2) = \mathcal{W}_{s_1,s_2}(D)|_{\mathcal{S}(U)}$. Thus
		$$
		\sum_{z \in \tilde{\pi}^{-1}(s_1s_2)} a_D(z) = \sum_{z \in  \mathcal{W}_{s_1,s_2}(D)|_{\mathcal{S}(U)}} a_D(z) = g_D|_{\mathcal{S}(U)}(s_1,s_2) = a_{D/U}(s_1,s_2)
		$$
		in this case as well. Finally, choose $w = s_1s_2 \cdots s_\ell \in \text{Im}(\pi)$. Then
		$$
		\pi^{-1}(w) = \{z'_1 z'_2 \cdots z'_{\ell - 1}s_\ell  \mid z'_is_{i + 1} \in \tilde{\pi}^{-1}(s_is_{i + 1}) \text{ for $i = 1, \dots, \ell - 1$}\},
		$$
		and from the preceding two calculations we have
		$$
		\sum_{z \in \pi^{-1}(w)} a_D(z) = \prod_{i = 1}^{\ell - 1}\Big( \sum_{z_i \in \tilde{\pi}^{-1}(s_is_{i + 1})} a_D(z'_i) \Big) = \prod_{i = 1}^{\ell - 1} a_{D/U}(s_i,s_{i + 1}) = a_{D/U}(w).
		$$
		This proves our second claim, and we conclude that
		$$
		g_{D/U}(s,t) = \sum_{w \in \text{Im}(\pi)} a_{D/U}(w) = \sum_{w \in \text{Im}(\pi)} \sum_{z \in \pi^{-1}(w)} a_D(z) = \sum_{z \in \mathcal{W}_{s,t}(D)} a_D(z) = g_D(s,t).
		$$
	\end{proof}

	Let $A_D$, $G_D$, $A_{D/U}$ and $G_{D/U}$ denote transition and Green's matrices for $D$ and $D/U$ respectively. Then the preceeding proposition and Lemma \ref{lem:path_conv} give us:
	
	\begin{cor}
		\label{cor:g_matrix}
		For any contraction as above, the Green's matrix
		$$
		G_{D/U} = (I - A_{D/U})^{-1} \in M_{|\mathcal{S}(D/U)|}(\Z(T))
		$$
		is well defined, and its entries coincide with the Green's function $g_{D/U}$ as in the case of tangle Markov chains.
	\end{cor}
	
	The next proposition gives a sufficient condition that ensures contracting preserves the determinant $\det(I - A_D)$. Cycles are defined as in Section \ref{sec:hikes}, and we abuse notation by letting $I$ denote identity matrices of different dimensions simultaneously.
	
	\begin{prop}
		\label{thm:cmc_det}
		If the contracted region $U$ admits no cycles, then
		$$
		\det(I - A_D) = \det(I - A_{D/U}).
		$$
		Equivalently, $\det(G_D) = \det(G_{D/U})$.
	\end{prop}

	An example of a tangle diagram which admits no cycles is a braid, since particles always move in the direction of orientation of the braid.

	\begin{proof}
		By the identity (\ref{eq:cf}), using the notation from Section \ref{sec:hikes}, it suffices to check that
		$$
		\sum_{q \in \mathcal{Q}(D/U)} (-1)^{|q|}(a_{D/U})_\text{circ}(q) = \sum_{q \in \mathcal{Q}(D)} (-1)^{|q|}(a_D)_\text{circ}(q).
		$$
		Let $\mathcal{Q}'(D)$ denote the set of all multicycles of $D$ which are {\em simple outside of $\mathcal{S}(D) - \mathcal{S}(D/U)$}. In other words, $\mathcal{Q}'(D)$ is the set of all finite sets of cycles $q' = \{c_1, \dots, c_k\}$ such that if a state $s$ appears multiple times in any one $c_i$, or appears in distinct $c_i$ and $c_j$, then $s \in \mathcal{S}(D) - \mathcal{S}(D/U)$. Note that since $U$ does not admit cycles, $\mathcal{Q}'(D)$ is finite.
		
		First, we claim
		\begin{equation}
			\label{eq:clone}
			\sum_{q \in \mathcal{Q}(D/U)} (-1)^{|q|}(a_{D/U})_\text{circ}(q) = \sum_{z \in \mathcal{Q}'(D)} (-1)^{|z|}(a_D)_\text{circ}(z).
		\end{equation}
		The proof of (\ref{eq:clone}) is analogous to the proof of Proposition \ref{thm:cmc}. Let $\widehat{\mathcal{Q}}(D/U)$ be the set of all finite sets $q = \{c_1, \dots, c_k\}$ such that:
		\begin{itemize}
			\item Each $c_i$ is a finite, non-empty sequence of states in $\mathcal{S}(D/U)$ considered up to cyclic permutation. That is, two such sequences are the considered the same if they are cyclic permutations of one another.
			\item No state appears in multiple $c_i$ or more than once in the same $c_i$.
		\end{itemize}
		Then $\widehat{\mathcal{Q}}(D/U)$ is finite and there is an obvious inclusion
		$$
		\mathcal{Q}(D/U) \hookrightarrow \widehat{\mathcal{Q}}(D/U).
		$$
		Furthermore, for any $\mathcal{Q}''$ with $\mathcal{Q}(D/U) \subset \mathcal{Q}'' \subset \widehat{\mathcal{Q}}(D/U)$, we have
		$$
		\sum_{q \in \mathcal{Q}(D/U)} (-1)^{|q|}(a_{D/U})_\text{circ}(q) = \sum_{q \in \mathcal{Q}''} (-1)^{|q|}(a_{D/U})_\text{circ}(q)
		$$
		since any element $ q \in \widehat{\mathcal{Q}}(D/U) - \mathcal{Q}(D/U)$ has $a_\text{circ}(q) = 0$. As in the proof of Proposition \ref{thm:cmc}, we define a map
		$$
		\pi : \mathcal{Q}'(D) \to \widehat{\mathcal{Q}}(D/U)
		$$
		as follows: given $z \in \mathcal{Q}'(D)$, we construct $\pi(z) \in \widehat{\mathcal{Q}}(D/U)$ by removing all states of $\mathcal{S}(D) - \mathcal{S}(D/U)$ from each cycle in $z$.
		
		Since $U$ does not admit cycles, no cycle of $D$ has empty image under $\pi$. Thus $\pi$ is well defined, and for all $z \in \mathcal{Q}'(D)$ we have $|z| = |\pi(z)|$. Additionally, the same arguments used in the proof of Proposition \ref{thm:cmc} show that
		$$
		\mathcal{Q}(D/U) \subset \text{Im}(\pi)
		$$
		and that for all $q \in \text{Im}(\pi)$,
		$$
		\sum_{z \in \pi^{-1}(q)}a_\text{circ}(z) = a_\text{circ}(q).
		$$
		Thus, as in the proof of Proposition \ref{thm:cmc}, we have
		\begin{align*}
		\sum_{q \in \mathcal{Q}(D/U)} (-1)^{|q|}(a_{D/U})_\text{circ}(q) &= \sum_{q \in \text{Im}(\pi)} (-1)^{|q|}(a_{D/U})_\text{circ}(q) \\
		&= \sum_{q \in \text{Im}(\pi)} (-1)^{|q|}\Big(\sum_{z \in \pi^{-1}(q)}a_\text{circ}(z)\Big) \\
		&= \sum_{q \in \text{Im}(\pi)} \Big(\sum_{z \in \pi^{-1}(q)}(-1)^{|z|}a_\text{circ}(z)\Big) =  \sum_{z \in \mathcal{Q}'(D)} (-1)^{|z|}(a_D)_\text{circ}(z),
		\end{align*}
		proving the claim.
		
		Next we claim
		\begin{equation}
			\label{eq:cltwo}
			\sum_{q \in \mathcal{Q}'(D)} (-1)^{|q|}(a_D)_\text{circ}(q) = \sum_{q \in \mathcal{Q}(D)} (-1)^{|q|}(a_D)_\text{circ}(q),
		\end{equation}
		which combined with (\ref{eq:clone}) proves the proposition. Let $\mathcal{Q}_\text{bad}'(D)$ be the set of all elements of $\mathcal{Q}'(D)$ which contain a repeated state. Equivalently,
		$$
		\mathcal{Q}_\text{bad}'(D) = \mathcal{Q}'(D) - \mathcal{Q}(D).
		$$
		Then equation (\ref{eq:cltwo}) is equivalent to the following lemma:
		\begin{lemma} 
			\label{lem:deferred}
			$$
				\sum_{q \in \mathcal{Q}_\text{bad}'(D)} (-1)^{|q|}(a_D)_\text{circ}(q) = 0.
			$$
		\end{lemma}
		Our proof of Lemma \ref{lem:deferred} is somewhat lengthy and has a different flavor then the rest of the paper, so we defer it to Section \ref{sec:lemma} below. Assuming the lemma, the proposition is proven.
	\end{proof}
	
	\section{Random walks on powers of full twists}
	\label{sec:burau}
	
	\subsection{Burau representation of powers of full twists}
	
	To study the behavior of $\rho_1$ under twisting, it will be important to understand random walks on powers of full twists. For this we take advantage of the fact that a full twist on $n$ coherently oriented strands is the braid $\Omega_n = (\sigma_1\sigma_2 \cdots \sigma_{n - 1})^n$ in the $n$-strand braid group $B_n$, where the $\sigma_i$ are the standard generators. For information on braid groups, see \cite{katu08}.
	
	\begin{conv}
		Throughout the paper we use $\Omega_n$ to indicate the full twist element of the $n$-strand braid group $B_n$. Additionally, we adopt the convention that braids are oriented vertically from bottom to top.
	\end{conv}
	
	Recall that the (unreduced) Burau representation $\psi : B_n \to \text{GL}_n(\Z[T,T^{-1}])$ is determined by 
	\begin{equation}
		\label{eq:sigma_rep}
		\psi(\sigma_k) = \begin{bmatrix}
			I_{k - 1} & 0 & 0 & 0 \\
			0 & 1 - T & T & 0 \\
			0 & 1 & 0 & 0 \\
			0 & 0 & 0 & I_{n - k - 1}
		\end{bmatrix} 
	\end{equation}
	for $k = 1, \dots, n - 1$. The connection with our random walk model is the following observation of Jones \cite{jon87}: if $s_i$ denotes the $i$th incoming strand of a braid $\beta \in B_n$ and $t_j$ denotes the $j$th outgoing strand, labeling from left to right in each case, then
	\begin{equation}
		\label{eq:braid_prob}
		\psi(\beta)_{ij} = g(s_i, t_j),
	\end{equation}
	where $g$ is the Green's function on the Markov chain induced by any diagram of $\beta$. Thus, we can understand random walks on $k$ full twists $\Omega_n^k$ by computing $\psi(\Omega_n^k)$. We begin by calculating $\psi(\Omega_n)$, assuming throughout the section that $n \geq 2$.
	
	\begin{prop}
		\label{thm:twist_rep}
		The Burau representation $\psi(\Omega_n)$ of a full twist on $n$ strands is defined by
		$$
			\psi(\Omega_n)_{ij} = \begin{cases}
				T^{j - 1}(1 - T) & i \neq j \\
				1 - \sum_{m = 1, m \neq i}^n \psi(\Omega_n)_{im} & i = j
			\end{cases} = \begin{cases}
			T^{j - 1}(1 - T) & i \neq j \\
			1 - (1 - T)(\sum_{m = 1, m \neq i}^n T^{m - 1})  & i = j
		\end{cases}.
		$$
	\end{prop}

	The $i = j$ case follows from the $i \neq j$ case and the observation that, for any braid $\beta \in B_n$, all entries in a given row of $\psi(\beta)$ sum to $1$. This is implied by the random walk interpretation and easily checked using the generating matrices in (\ref{eq:sigma_rep}). When $i \neq j$ we use the fact that $\Omega_n$ is central in $B_n$ (in fact $\Omega_n$ generates the center of $B_n$), so $\psi(\Omega_n)$ commutes with $\psi(\beta)$ for all $\beta \in B_n$.
	
	\begin{lemma}
		\label{thm:central_rep}
		Suppose $A = (a_{ij}) \in \text{GL}_n(\Z[T,T^{-1}])$ commutes with $\psi(\beta)$ for all $\beta \in B_n$. Then there exists a polynomial $p_A \in \Z[T,T^{-1}]$ such that for all $i \neq j$,
		$$
		a_{ij} = T^{j - 1}p_A.
		$$
	\end{lemma}

	\begin{proof}
		Let $r_i$ denote the $i$th row of $A$ and $c_j$ the $j$th column. Using (\ref{eq:sigma_rep}), we compute
		$$
		\psi(\sigma_k)A = \begin{bmatrix}
			r_1 \\
			\vdots \\
			r_{k - 1} \\
			(1 - T)r_k + Tr_{k + 1} \\
			r_k \\
			r_{k + 2}\\
			\vdots \\
			r_n
		\end{bmatrix}
		$$
		and
		$$
		A \psi(\sigma_i) = [c_1, \cdots, c_{k - 1}, (1 - T)c_k + c_{k + 1}, Tc_k, c_{k + 2}, \cdots, c_n]
		$$ 
		for all $k = 1, \dots, n - 1$. Let $M_k = \psi(\sigma_k)A = A \psi(\sigma_k)$; then we obtain three sets of equations:
		\begin{align*}
			a_{k,j} &= (M_k)_{k+1,j} = a_{k + 1,j} &\text{for } j \neq k, k + 1 \\
			a_{i,k+1} &= (M_k)_{i,k+1} = Ta_{i,k} &\text{for } i \neq k, k + 1 \\
			Ta_{k+1,k} &= (M_k)_{k,k} - (1 - T)a_{k,k} = a_{k,k+1} \\
		\end{align*}
		The first set of equations says that two adjacent, off-diagonal elements of a given column of $A$ are equal, and the second says adjacent, off-diagonal row elements are related by multiplying by $T$. The third set of equations says the two elements $a_{k + 1,k}$ and $a_{k,k + 1}$, adjacent to the main diagonal of $A$ on opposite sides, are also related by multiplying by $T$. It follows that all off-diagonal entries of $A$ are determined by the polynomial $p_A = a_{n1}$ in the desired way.
	\end{proof}

	\begin{proof}[Proof of Proposition \protect\ref{thm:twist_rep}]
		By Lemma \ref{thm:central_rep} and the preceding discussion, the entries of $\psi(\Omega_n)$ are completely determined by the polynomial
		$$
		p_{\psi(\Omega_n)} = \psi(\Omega_n)_{n1}.
		$$
		To compute $\psi(\Omega_n)_{n1}$, we use the probabilistic interpretation of $\psi(\Omega_n)$: we claim a particle which enters $\Omega_n$ on the $n$th incoming strand has a unique random walk by which it can exit $\Omega_n$ on the first strand, and this walk has probability $1 - T$. This claim can be checked by drawing a picture, which we leave to the reader---see Figure \ref{fig:full_twist} for the case $n = 3$. It then follows from (\ref{eq:braid_prob}) that $\psi(\Omega_n)_{n1} = 1 - T$, as the proposition claims.
	\end{proof}

	Having computed $\psi(\Omega_n)$, it is straightforward to determine $\psi(\Omega_n^k)$ for all $k \in \N$. By Lemma \ref{thm:central_rep} and the preceeding discussion, since $\Omega^k_n$ is central, there exists a family of polynomials $p_{n,k} \in \Z[T,T^{-1}]$ such that $p_{n,1} = 1 - T$ for all $n$, and
	\begin{equation}
		\label{eq:pre_twist_pow_rep}
		\psi(\Omega^k_n)_{ij} = \begin{cases}
			T^{j - 1}p_{n,k} & i \neq j \\
			1 - p_{n,k}(\sum_{m = 1, m \neq i}^n T^{m - 1}) & i = j
		\end{cases}
	\end{equation}
	for all $n$ and $k$.
	
	\begin{lemma}
		\label{thm:twist_pow_q}
		With notation as above,
		$$
		p_{n,k} = (1 - T)\sum_{i = 0}^{k - 1} T^{ni}.
		$$
	\end{lemma}

	\begin{proof}
		Proposition \ref{thm:twist_rep} handles the $k = 1$ case for all $n$, and the proof proceeds by induction on $k$. Using (\ref{eq:pre_twist_pow_rep}), we compute
		\begin{align*}
			p_{n,k+1} &= \psi(\Omega_n^{k+1})_{n1} \\
			&= \sum_{i = 1}^n \psi(\Omega_n^k)_{ni}\psi(\Omega_n)_{i1} \\
			&= \psi(\Omega_n^k)_{n1}\psi(\Omega_n)_{11} + \sum_{i = 2}^{n - 1} \psi(\Omega_n^k)_{ni}\psi(\Omega_n)_{i1} + \psi(\Omega_n^k)_{nn}\psi(\Omega_n)_{n1} \\
			&= p_{n,k}(1 - T + T^n) + p_{n,k}(T - T^{n - 1}) + (1 - T - p_{n,k}(1 - T^{n - 1})) \\
			&= 1 - T + T^np_{n,k},
		\end{align*}
		which is the correct sum.
	\end{proof}

	We summarize our findings as follows:
	
	\begin{prop}
		\label{thm:twist_pow_rep}
		For all $k \geq 1$ and $n \geq 2$, $\psi(\Omega^k_n)_{ij}$ is given by (\ref{eq:pre_twist_pow_rep}), where $p_{n,k}$ is the polynomial of Lemma \ref{thm:twist_pow_q}.
	\end{prop}
	
	\subsection{Asymptotic behavior}
	\label{sec:ab}

	To describe the limiting behavior of $p_{n,k}$ as $k \to \infty$, we introduce the following terminology.
	
	\begin{defn}
		\label{def:stab}
		Let $\{p_t\}_{t \in \N}$ be a sequence of Laurent polynomials in $\Z[T,T^{-1}]$, and let $(p_t)_r$, $r \in \Z$, denote the coefficient of $T^r$ in $p_t$. We say $\{p_t\}$ {\em stabilizes positively} if for any $r_0 \in \Z$ there exists $t_0 \in \N$, such that for all $t \geq t_0$ and all $r \leq r_0$,
		$$
		(p_t)_r = (p_{t_0})_r.
		$$
		If $\{A_t\}_t = \{(a(t)_{ij})\}_t$ is a sequence of matrices in $M_n(\Z[T,T^{-1}])$, we say $\{A_t\}$ {\em stabilizes positively} if the sequences $\{a(t)_{ij}\}_t$ do so for all $i$ and $j$. In either case, we write $\lim_{t \to \infty} p_t$ to denote the limiting Laurent series or matrix of series.
	\end{defn}

	If we consider sequences of polynomials in $\Z[T]$ rather than $\Z[T,T^{-1}]$, then positive stabilization is just convergence in the $T$-adic norm on $\Z[[T]]$. It is clear that the $p_{n,k}$ of Lemma \ref{thm:twist_pow_q} stabilize positively as $k \to \infty$, and the limit is given explicitly by
	\begin{equation}
		\label{eq:limit_poly}
		\lim_{k \to \infty} p_{n,k} = (1 - T)\sum_{i = 0}^\infty T^{ni} = \frac{1 - T}{1 - T^n} = \frac{1}{1 + T + \cdots + T^{n - 1}}.
	\end{equation}

	More generally, we have the following elementary lemma.
	
	\begin{lemma}
		\label{lem:stab_cont}
		If $\{p_t\}_{t \in \N}$ and $\{q_t\}_{t \in \N}$ are two sequences of Laurent polynomials which stabilize positively, then the sequences $\{p_t + q_t\}_t$ and $\{p_tq_t\}_t$ stabilize positively as well. Furthermore,
		$$
		\lim_{t \to \infty} (p_t + q_t) = \lim_{t \to \infty} p_t + \lim_{t \to \infty} q_t
		$$
		and
		$$
		\lim_{t \to \infty} (p_tq_t) = (\lim_{t \to \infty} p_t)(\lim_{t \to \infty} q_t).
		$$
	\end{lemma}
	
	\begin{proof}
		It is easy to check that $\{p_t + q_t\}_t$ stabilizes positively and has the desired limit.
		
		For the product sequence $\{p_tq_t\}$, given $p \in \Z[T,T^{-1}]$, let $(p)_r$ denote the coefficient of $T^r$ in $p$. By definition, we may choose $t_0$ such that for all $t \geq t_0$ and all $r \leq 0$, $(p_t)_r = (p_{t_0})_r$ and $(q_t)_r = (q_{t_0})_r$. If $p_{t_0}$ contains terms with negative degree, let $m_p$ be the minimal degree among these. Otherwise let $m_p = 0$. Define $m_q$ analogously for $q_{t_0}$, and set $m = \min(m_p,m_q)$. By our choice of $t_0$, for all $t \geq t_0$ and all $r < m$,
		$$
		(p_t)_r = (q_t)_r = 0.
		$$
		Now fix $r_1 \in \Z$ arbitrarily. Choose $t_1 > t_0$ such that the first $r_1 + |m|$ coefficients of $p_t$ and $q_t$ have stabilized for all $t \geq t_1$. Then for all $t \geq t_1$ and all $r \leq r_1$, by the above equation,
		$$
		(p_tq_t)_r = \sum_{i = m}^{r + |m|} = (p_t)_i(q_t)_{r - i} = \sum_{i = m}^{r + |m|} = (p_{t_1})_i(q_{t_1})_{r - i} = (p_{t_1}q_{t_1})_r.
		$$
		Thus the sequence $\{p_tq_t\}_t$ stabilizes positively with the desired limit.
	\end{proof}
	
	\begin{prop}
		\label{thm:twist_stab}
		For all fixed $n \geq 2$, the sequence of matrices $\{\psi(\Omega_n^k)\}_{k \in \N}$ stabilizes positively as $k \to \infty$. The limit is formally equivalent to a matrix of rational functions in $T$, given explicitly by
		\begin{equation}
			\label{eq:inf_prob}
			\psi(\Omega^\infty_n)_{ij} = \lim_{k \to \infty} \psi(\Omega_n^k)_{ij} = \frac{T^{j - 1}}{1 + T + \cdots + T^{n - 1}}.
		\end{equation}
	\end{prop}

	\begin{proof}
		Since the sequence $\{p_{n,k}\}_k$ of \Cref{thm:twist_pow_q} stabilizes positively for all fixed $n$, the positive stabilization of the sequence $\{\psi(\Omega_n^k)\}_{k \in \N}$ follows immediately from Proposition \ref{thm:twist_pow_rep} and Lemma \ref{lem:stab_cont}. The last part of the theorem follows from Lemma \ref{lem:stab_cont} as well, allowing us to substitute (\ref{eq:limit_poly}) for $p_{n,k}$ in (\ref{eq:pre_twist_pow_rep}) to compute the limit.
	\end{proof}

	Returning to the random walks model, we have:
	
	\begin{cor}
		\label{thm:prob_stab}
		For fixed $n \geq 2$ and $k \geq 1$, let $s^k_i$ and $t^k_j$ denote the $i$th incoming and $j$th outgoing strand respectively of the braid $\Omega_n^k$ consisting of $k$ full positive twists. Then the sequence $\{g(s^k_i,t^k_j)\}_{k \in \N}$ stabilizes positively for all $i$ and $j$.
	\end{cor}

	\begin{rmk}
		The expression (\ref{eq:inf_prob}) can be interpreted amusingly (and non-rigorously) as the probability that a bowling ball traversing an infinite torus braid will exit at the $j$th outgoing strand, and it is interesting to note that this probability does not depend on where the ball enters the braid. This lack of $i$ dependence can be explained intuitively as follows: since the braid is infinite, if the ball visits the $i$th strand of the braid then it may as well have entered on that strand. Eventually the ball will visit every strand, so its starting point does not matter.
	\end{rmk}
	
	\begin{rmk}
		\label{rem:roz}
		We have shown here that the Burau representation of an infinite torus braid can be well defined as a limit. Although we were unable to find these results presented elsewhere in the literature, we note that the related {\em Temperley-Leib} representation of an infinite torus braid is also known to be a well-defined limit---as Rozansky observes, this limit coincides with the Jones-Wenzl idempotent \cite{roz14}. In this context, our results here are not surprising.
	\end{rmk}
	
	\section{Stabilization of the Alexander polynomial}
	\label{sec:alex}
	
	Since the $\rho_1$ invariant is normalized by the Alexander polynomial, understanding the asymptotics of the latter is necessary for proving Theorem \ref{thm:main}. Fortunately the behavior of the Alexander polynomial under twisting is well understood. Let $K \subset S^3$ be a knot and $L \subset S^3$ an unknot such that $K \cup L$ is a two-component link with linking number $n \neq 0$. Let $K_t$ be the knot resulting from performing $1/t$ surgery on $L$; then the classical {\em Torres formulas} imply that
	\begin{equation}
		\label{eq:torres}
		\Delta_{K_t}(T) = \frac{T - 1}{T^n - 1} \Delta_{K \cup L}(T,T^{tn}) = \frac{\Delta_{K \cup L}(T,T^{tn})}{1 + T + \dots + T^{n - 1}},
	\end{equation}
	where $\Delta_{K \cup L}$ here indicates the {\em multi-variable Alexander polynomial} with first variable corresponding to $K$ and second corresponding to $L$ \cite{tor53, sw04, bamo17-2}.

	\begin{lemma}
		\label{lem:alex}
		Let $\{K_t\}$ be a coherently oriented family of knots twisted along $n$ strands. Then the sequence $\{T^{tn(n -1)/2}\Delta_{K_t}(T)\}_t$ stabilizes positively to a rational function as $t \to \infty$.
	\end{lemma}
	
	We give two proofs of \Cref{lem:alex}. The first uses the identity (\ref{eq:torres}) and we omit details for the sake of space, and the second proof uses the machinery developed in the preceeding two sections. (A third proof of the stabilization of the Alexander polynomial under twisting can be found in \cite{che22}.) Though the second argument is significantly more involved, we present it as a preview of the techniques used in \Cref{sec:rho} to work with $\rho_1$.
	
	\begin{proof}[First proof (sketch).]
		Using the notation of (\ref{eq:torres}), we first observe that the sequence
		$$
		\{\Delta_{K \cup L}(T,T^{tn})\}_t
		$$
		stabilizes positively. Indeed, let $m$ be the minimum degree of any monomial of $\Delta_L(T,T^n)$, and let $\big( \Delta_{K \cup L}(T,T^{tn}) \big)_r \in \Z$ denote the coefficient of $T^r$ in $\Delta_{K \cup L}(T,T^{tn})$. Then for arbitrary $t_0$ and all $t > t_0$, it is easy to check that
		$$
		\big( \Delta_{K \cup L}(T,T^{tn}) \big)_r = \big( \Delta_{K \cup L}(T,T^{t_0n}) \big)_r
		$$
		for all $r < t_0n + m$. In the limit the terms with $t$ in the exponent escape to infinity, and we have
		$$
		\lim_{t \to \infty} \Delta_{K \cup L}(T,T^{tn}) = \Delta_{K \cup L}(T,0).
		$$
		It follows from this, from (\ref{eq:torres}), and from continuity properties similar to those of \Cref{lem:stab_cont} that the sequence of symmetrized Alexander polynomials $\{\Delta_{K_t}\}$ satisfies
		\begin{equation}
			\label{eq:easy}
			\lim_{t \to \infty} \{\Delta_{K_t}\}_t = \frac{\Delta_{K \cup L}(T,0)}{1 + T + \dots + T^{n - 1}}
		\end{equation}
		up to normalization by a power of $T$. With some additional work, one can show that the correct normalization is multiplication by $T^{tn(n -1)/2}$ on the left side.
	\end{proof}
	
	The following definition will be needed in our second proof of \Cref{lem:alex}, and in \Cref{sec:rho}.
	
	\begin{defn}
		\label{def:compat}
		Let $\{K_t\}_{t \geq 0}$ be a coherently oriented family of knots in $S^3$, twisted on $n$ strands. We call a sequence of long knot diagrams $\{D_t\}_{t \geq 0}$ {\em compatible} with the family if:
		\begin{itemize}
			\item For all $t$, $D_t \subset I^2$ is a long knot diagram of $K_t$.
			\item $D_t$ is obtained by inserting $t$ full twists on $n$ parallel strands in a fixed region $U_0 \subset D_0$.
		\end{itemize}
		We call the region $U_t \subset D_t$ replacing $U_0$ the {\em twisted region} of $D_t$.
	\end{defn}

	It is clear that a compatible family of diagrams can be found for any twisted family of knots. Additionally, Definition \ref{def:compat} implies that for any such family of diagrams $\{D_t\}$ with twisted regions $U_t \subset D_t$, the subdiagrams $D_t - U_t$ are identical for all $t$. The twisted region $U_t$ consists of the braid $\Omega^t_n$ of Section \ref{sec:burau}.
	
	\begin{proof}[Second proof of \protect\Cref{lem:alex}.]
		Let $\{D_t\}$ be a family of long knot diagrams compatible with the $K_t$, and let $U_t$ denote the twisted region of $D_t$ for all $t$. Let $D_t/U_t$ be the Markov chain formed by contracting $D_t$ by $U_t$, as in Definition \ref{def:cmc}, and let $A_t$ and $A_t^U$ be transition matrices for $D_t$ and $D_t/U_t$ respectively. Since $U_t$ admits no cycles, (\ref{eq:alex}) and Proposition \ref{thm:cmc_det} give
		$$
		T^{(\varphi(D_t) + w(D_t))/2}\Delta_{K_t}(T) = \det(I - A_t) = \det(I - A^U_t)
		$$
		for all $t$. Additionally $\varphi(D_t) = \varphi(D_0)$ for all $t$, and since a full twist on $n$ strands has $n(n - 1)$ positive crossings,
		$$
		w(D_t) = w(D_0) + tn(n - 1).
		$$
		Thus
		\begin{equation}
			\label{eq:alex_seq}
			T^{tn(n - 1)/2} \Delta_{K_t}(T) = T^{(-\varphi(D_0) - w(D_0))/2}\det(I - A^U_t).
		\end{equation}
		Since $D_t - U_t$ is identical to $D_0 - U_0$ for all $t$, we have natural identifications
		$$
		\mathcal{S}(D_t/U_t) \leftrightarrow \mathcal{S}(D_0 - U_0)
		$$
		and the matrices $\{A^U_t\}$ are all of the same dimension. We also assume the states $\mathcal{S}(D_t/U_t)$ have been ordered using a fixed ordering of $\mathcal{S}(D_0 - U_0)$ for all $t$. To prove the lemma, by (\ref{eq:alex_seq}) and Lemma \ref{lem:stab_cont}, it suffices to show the sequence $\{A^U_t\}_t$ stabilizes positively to a matrix of rational functions.
		
		We choose $i$ and $j$ indexing states $s_i$ and $s_j$ of $D_0 - U_0$, which we identify with elements of $	\mathcal{S}(D_t/U_t)$ for all $t$. By definition, if $s_i$ is not an incoming strand of $U_t$ or $s_j$ is not an outgoing one, then
		\begin{equation}
			\label{eq:first_case}
			(A^U_t)_{ij} = a_{D_t/U_t}(s_i,s_j) = a_{D_0}(s_i,s_j)
		\end{equation}
		and the sequence $\{(A^U_t)_{ij}\}$ is constant. Otherwise
		\begin{equation}
			\label{eq:second_case}
			(A^U_t)_{ij} = a_{D_t/U_t}(s_i,s_j) = g|_{\mathcal{S}(U_t)}(s_i,s_j) = \psi(\Omega_n^t)_{k_i,k_j} 
		\end{equation}
		where $\psi$ is the Burau representation as in Section \ref{sec:burau}, $s_i$ is the $k_i$th incoming strand of $\Omega^t_n$, and $s_j$ is the $k_j$th outgoing one. By Proposition \ref{thm:twist_stab} the sequence $\{\psi(\Omega_n^t)_{k_i,k_j}\}_t$ stabilizes positively, so $\{(A^U_t)_{ij}\}_t$ stabilizes positively for all $i$ and $j$. Thus $\{A_t^U\}_t$ stabilizes positively, and so does the Alexander polynomial.
		
		Considering the limit in the first case (\ref{eq:first_case}), we have
		$$
		\lim_{t \to \infty} (A^U_t)_{ij} = a_{D_0}(s_i,s_j),
		$$
		and in (\ref{eq:second_case}) by Proposition \ref{thm:twist_stab} we have
		$$
		\lim_{t \to \infty} (A^U_t)_{ij} = \lim_{t \to \infty} \psi(\Omega_n^t)_{k_i,k_j} = \frac{T^{k_j - 1}}{1 + T + \cdots + T^{n - 1}}.
		$$
		Since the limit is a rational function in both cases,
		\begin{equation}
			\label{eq:last_alex}
			\lim_{t \to \infty}T^{tn(n - 1)/2} \Delta_{K_t} = T^{-(\varphi(D_0) + w(D_0))/2}\det(I - \lim_{t \to \infty} A^U_t)
		\end{equation}
		is rational.
	\end{proof}

	Both proofs of \Cref{lem:alex} show how to easily calculate the limit $\lim_{t \to \infty} \Delta_{K_t}$ for any coherently oriented, twisted family of knots. To further aid computation we introduce the notion of an {\em infinite twist vertex}, which will be useful in Section \ref{sec:rho} as well.
	
	\begin{defn}
		\label{def:twist_vert}
		An {\em infinite twist vertex with in-degree $n$}, $n \geq 2$, is a vertex drawn in the plane with $n$ adjacent incoming edges, temporarily labeled $s_1, \dots, s_n$ from left to right, and $n$ adjacent outgoing edges labeled $t_1, \dots, t_n$. This vertex, together with its $2n$ edges, is interpreted as a Markov chain (or part of a larger tangle Markov chain) via the following transition function rules for all $1 \leq i,j \leq n$:
		\begin{align*}
			a(s_i,s_j) &= a(t_i,t_j) = 0 \\
			a(s_i, t_j) &= \frac{T^{j - 1}}{1 + T + \cdots + T^{n - 1}}.
		\end{align*}
	\end{defn}
	As the name suggests, an infinite twist vertex is meant to represent a twist region containing infinitely many full twists---compare the above values with those in Proposition \ref{thm:twist_stab}. We use infinite twist vertices to define two diagrams associated to any compatible family.
	
	\begin{defn}
		\label{def:spec_diags}
		Let $\{K_t\}$ be a coherently oriented family of knots twisted on $n$ strands, and let $\{D_t\}$ be a compatible family of diagrams as in \Cref{def:compat}. In this context, let $D_\infty$ be the diagrammatic Markov chain obtained by replacing the twisted region $U \subset D_0$ with an infinite twist vertex with in-degree $n$. Additionally, let $D^\tau_\infty$ be the Markov chain constructed by replacing $U$ with an infinite twist vertex followed by a full twist on $n$ strands and then a second infinite twist vertex. We label the full twist between the two infinite twist vertices in the latter case by $\tau_\infty \subset D^\tau_\infty$.
	\end{defn}

	See Figures \ref{fig:d_inf} and \ref{fig:d_tau_inf} for examples of $D_\infty$ and $D^\tau_\infty$ for the family of $(2,q)$-torus knots. The Markov chain $D^\tau_\infty$ will be used in Section \ref{sec:rho}.
	
	Let $\{K_t\}$ be a coherently oriented family of knots twisted on $n$ strands, let $\{D_t\}$ be a set of compatible diagrams with twisted regions $U_t$, and let $D_\infty$ be as in \Cref{def:spec_diags}. Let $A^U_t$ be a transition matrix for $D_t/U_t$, as in the second proof of \Cref{lem:alex}, and let $A^U_\infty$ be the transition matrix for $D_\infty$ determined by the same ordering on $\mathcal{S}(D_0 - U_0)$. Comparing the limits in the second proof of \Cref{lem:alex} with the values in Definition \ref{def:twist_vert}, we find that
	$$
	\lim_{t \to \infty} A^U_t = A^U_\infty.
	$$
	This and (\ref{eq:last_alex}) imply:
	
	\begin{cor}
		With notation as above,
		\label{cor:pre_lim_alex}
		$$
			\lim_{t \to \infty}T^{tn(n - 1)/2} \Delta_{K_t} = T^{(-\varphi(D_0) -w(D_0))/2}\det(I - A^U_\infty).
		$$
	\end{cor}
	As an example we consider the family $K_t = {\bf T}(2,2t + 1)$, $t \geq 0$, where {\bf T}$(p,q)$ is the $(p,q)$-torus knot. For this family the diagram $D_\infty$ can be drawn as in Figure \ref{fig:d_inf}, where the vertex is an infinite twist vertex. Then
	$$
	A^U_\infty = \begin{bmatrix}
		0 & T &0 & 1 - T & 0 \\
		0 & 0 & \frac{T}{1 + T} & 0 & \frac{1}{1 + T}  \\
		0 & 0 & 0 & 1 & 0 \\
		0 & 0 & \frac{T}{1 + T}  & 0 & \frac{1}{1 + T} \\
		0 & 0 & 0 & 0 & 0
	\end{bmatrix}.
	$$
	Since $\varphi(D_0) = -1$ and $w(D_0) = 1$, we compute
	$$
		\lim_{t \to \infty}T^{tn(n - 1)/2} \Delta_{K_t} = \det(I - A^U_\infty) = \frac{1}{T + 1} = 1 - T + T^2 - T^3 + \cdots
	$$
	This is indeed the limit of the Alexander polynomials of $(2,2t + 1)$-torus knots as $t \to \infty$, as can be verified using (\ref{eq:easy}).
	
	\begin{figure}[t]
		\labellist
		\small \hair 2pt
		\pinlabel $U_0$ at -24 180
		\pinlabel $\tau_\infty$ at 585 213
		\endlabellist
		
		\subcaptionbox{$D_0$ \label{fig:d_zero}}{
			\includegraphics[height=3.5cm]{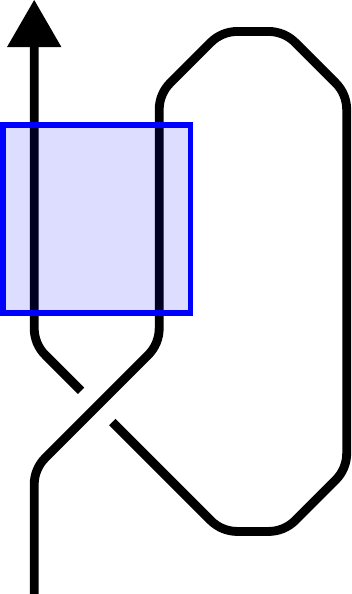}
		}
		\hspace{1cm}
		\subcaptionbox{$D_\infty$ \label{fig:d_inf}}{
			\includegraphics[height=3.5cm]{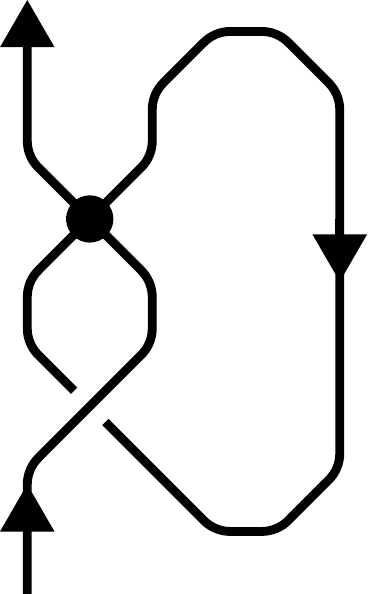}
		}
		\hspace{1cm}
		\subcaptionbox{$D^\tau_\infty$ \label{fig:d_tau_inf}}{
			\includegraphics[height=4.5cm]{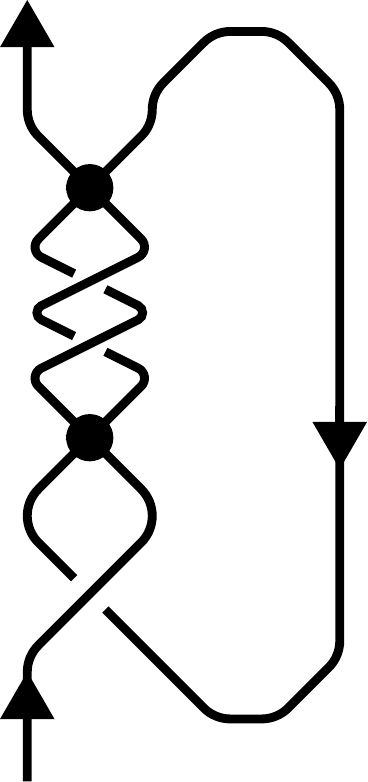}
		}
		\caption{$D_0$, $D_\infty$ and $D^\tau_\infty$ for the family of $(2,2t + 1)$-torus knots}
	\end{figure}

	We now apply this computation technique to prove the following lemma, which will be needed in the next section.

	\begin{lemma}
		\label{thm:alex_limit}
		Let $\{K_t\}$ be a family of knots twisted on $n$ coherently oriented strands. Then the limit of Alexander polynomials, considered as a rational function in $\Z(T)$, satisfies
		$$
		(\lim_{t \to \infty} \Delta_{K_t})|_{T = 1} = \pm \frac{1}{n}.
		$$
		In particular, $\lim_{t \to \infty} \Delta_{K_t}$ is not a polynomial.
	\end{lemma}
	
	The above limit of Alexander polynomials for the $(2,2t + 1)$-torus knots demonstrates Lemma \ref{thm:alex_limit}. Though \Cref{thm:alex_limit} can be proven using classical methods, we give a proof using infinite twist vertices---this technique will be echoed in the proof of Proposition \ref{thm:rho_limit} below.
	
	\begin{proof}
		Let $D_t$ be a family of compatible diagrams for the $K_t$, and let $D_\infty \subset I^2$ be as in Definition \ref{def:spec_diags}. We may think of $D_\infty$ as a planar diagram of a directed graph properly embedded in the unit cube---this graph has one interior vertex $v$, the infinite twist vertex, with $n - 1$ loops attached plus one additional incoming edge and one outgoing one. Let $A_\infty$ be a transition matrix for $D_\infty$; then by \Cref{cor:pre_lim_alex} we have
		$$
		(\lim_{t \to \infty} \Delta_{K_t})|_{T = 1} = \pm\det(I - A_\infty|_{T = 1}).
		$$
		
		When $T = 1$ in the Markov chain, particles traversing $D_\infty$ never ``jump down'' when passing over crossings. Equivalently, an undercrossing can be switched to an overcrossing without changing the matrix $A_\infty|_{T = 1}$. After performing a sequence of such crossing switches, we may assume that each loop and edge attached to $v$ in $D_\infty$ is unknotted from itself and from the other loops and edges. Additionally, being a limit of Alexander polynomials, $\det(I - A_\infty)$ is unchanged by Reidemeister moves on $D_\infty$. Moreover, from Definition \ref{def:twist_vert} we see that when $T = 1$, every nonzero weight associated with the infinite twist vertex is $1/n$. It follows that incoming or outgoing strands of the infinite twist vertex can be reordered without changing $A_\infty|_{T = 1}$. After performing a sequence of such reorderings and Reidemeister moves, we assume our diagram $D_\infty$ has no crossings as in Figure \ref{fig:crossingless_d_inf}. At $T = 1$ this diagram has transition matrix
		$$
		A|_{T = 1} = \frac{1}{n} \begin{bmatrix}
			\vdots & {\bf 1}_n \\
			0 & \cdots \\
		\end{bmatrix},
		$$
		where ${\bf 1}_n$ is the $n$-by-$n$ constant matrix of all $1$s, and the vertical and horizontal dots indicate a column and a row of $0$s respectively. Thus
		\begin{align*}
		\det (I - A|_{T = 1}) &= \det \begin{bmatrix}
			1 & -\frac{1}{n} & -\frac{1}{n} & \cdots & -\frac{1}{n} \\
			0 & \frac{n - 1}{n} & -\frac{1}{n} & \cdots & -\frac{1}{n} \\
			0 & -\frac{1}{n} & \frac{n - 1}{n} & \cdots & -\frac{1}{n} \\
			\vdots & \vdots & \vdots & \ddots & \vdots \\
			0 & 0 & 0 & \cdots & 1
		\end{bmatrix} \\
		&= \pm\frac{1}{n^{n - 1}} \det \begin{bmatrix}
			1 - n & 1 & \cdots & 1 \\
			1 & 1 - n & \cdots & 1 \\
			\vdots & \vdots & \ddots & \vdots \\
			1 & 1& \cdots &  1 - n
		\end{bmatrix}.
		\end{align*}
		The third matrix above is the result of removing the first and last row and column from the second matrix and factoring out $-1/n$ from each column---it is an $(n - 1)$-by-$(n - 1)$ matrix with $1 - n$ on the main diagonal and $1$ elsewhere. Denote this matrix by $B$. The determinant of $B$ is, up to a sign, the characteristic polynomial $p_{{\bf 1}_{n -1}}(\lambda)$ of the $(n - 1)$-by-$(n - 1)$ matrix ${\bf 1}_{n - 1}$ evaluated at $\lambda = n$. The matrix ${\bf 1}_{n - 1}$ has two eigenvalues, $0$ and $n - 1$, with multiplicity $n - 2$ and $1$ respectively. Thus
		$$
		p_{{\bf 1}_{n -1}}(\lambda) = \lambda^{n - 2}(\lambda - n + 1)
		$$
		and
		$$
		(\lim_{t \to \infty} \Delta_{K_t})|_{T = 1} = \pm\frac{1}{n^{n - 1}} \det B = \pm\frac{p_{{\bf 1}_{n-1}}(n)}{n^{n - 1}} = \pm\frac{1}{n}
		$$
		as desired.
	\end{proof}
	
	\begin{figure}[t]
		\includegraphics[height=4cm]{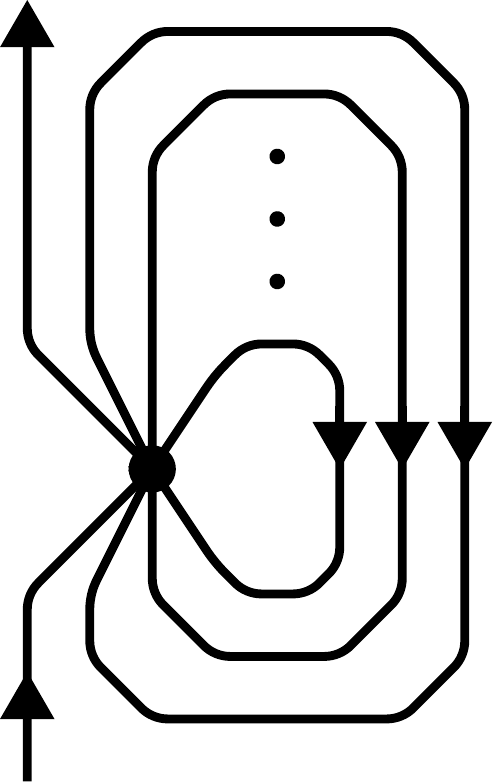}
		\caption{A crossingless diagram with an infinite twist vertex}
		\label{fig:crossingless_d_inf}
	\end{figure}

	\section{Stabilization of the $\rho_1$ growth rate}
	\label{sec:rho}

	We are now prepared to prove Theorems \ref{thm:main} and \ref{thm:example}---the arguments here are technical, but are essentially more nuanced versions of those in the previous section. We first recall that, for any matrix $A \in \text{M}_n(\Z(T))$, the $i$-$j$ {\em cofactor} of $A$ is $(-1)^{i + j}$ times the determinant of the matrix formed by removing the $i$th row and $j$th column from $A$. The {\em cofactor matrix} of $A$ is the matrix whose $i$-$j$ entry is the $i$-$j$ cofactor of $A$, and the {\em adjugate matrix} adj$(A)$ is the transpose of the cofactor matrix of $A$. It is a classical fact that if $\det(A) \neq 0$, then adj$(A)$ satisfies
	$$
	\text{adj}(A) = \det(A) \cdot A^{-1}.
	$$

	To avoid working with rational functions, we find it useful to redefine the invariant $\rho_1$ using adjugate matrices. To this end let $D$ be a long knot diagram of a knot $K \subset S^3$, with crossings $\mathcal{C}$ and states $\mathcal{S}$. Let $A$ be a transition matrix for $D$, let
	$$
	G = (g_{ij}) = (I - A)^{-1}
	$$
	be the Green's matrix, and let $\widetilde{G}$ be the adjugate matrix
	$$
	\widetilde{G} = (\tilde{g}_{ij}) = \text{adj}(I - A) = \det(I - A) G.
	$$
	
	As in Section \ref{sec:ud} we write a crossing $c \in \mathcal{C}$ as a triple $(\sigma, i, j)$, where $\sigma = \pm1$ is the sign of $c$ and $i$ and $j$ are the respective labels of the incoming over- and under-strands. We also denote the label of the outoing under-strand of $c$ by $j^+$. For such a crossing $c$ we define
	\begin{align}
		\widetilde{R}_1(c) &= \det(I - A)^2 R_1(c) \nonumber \\
		\label{eq:new_r_adj}
		&= \sigma(\tilde{g}_{ji}(\tilde{g}_{j^+, j} + \tilde{g}_{j, j^+} - \tilde{g}_{ij}) - \tilde{g}_{ii}(\tilde{g}_{j,j^+} - \det(I - A)) - \det(I - A)^2/2).
	\end{align}
	
	For $k \in \N$ indexing a state $s_k \in \mathcal{S}$, we also define
	\begin{equation}
		\label{eq:newphi}
		\tilde{\varphi}(s_k) = \tilde{\varphi}(k) =  \det(I - A)^2 \varphi_k (g_{kk} - 1/2) = \det(I - A) \varphi_k \tilde{g}_{kk} - \varphi_k\det(I - A)^2/2.
	\end{equation}
	Here, as in Section \ref{sec:ud}, $\varphi_k$ is the turning number of strand $k$. By (\ref{eq:alex}) and Definition \ref{def:rho},
	\begin{equation}
		\label{eq:new_rho}
		\rho_1(K) = T^{-\varphi(D) - w(D)} \Big( \sum_{c \in \mathcal{C}} \tilde{R}_1(c) - \sum_{s_k \in \mathcal{S}} \tilde{\varphi}(k) \Big).
	\end{equation}

	A key point here is that each term in (\ref{eq:new_r_adj}) is a polynomial expression in the entries of $A$---this will allow us to apply the continuity results of \Cref{lem:stab_cont}.
	\begin{lemma}
		\label{lem:polynomial}
		For every positive $m$, any choice of indices $1 \leq i, j, j^+ \leq m$ and $\sigma \in \{-1,1\}$, there exists a unique polynomial function $f_{m,\sigma,i,j,j^+}$ in the entries of $m$-by-$m$ matrices such that if $A \in M_m(\Z(T))$ is a transition matrix for a tangle Markov chain and $c = (\sigma, i, j, j^+)$ is a crossing of the tangle, then
		$$
		f_{m,\sigma,i,j,j^+}(A) = \widetilde{R}_1(c).
		$$
		The function $f$ is defined explicitly by (\ref{eq:new_r_adj}).
	\end{lemma}

	We fix some notation for the rest of the section. Let $\{K_t\}_{t \in \N}$ be a family of knots twisted around $n$ coherently oriented strands, let $D_t$ be a compatible family of diagrams as in \Cref{def:compat}, and let $U_t$ be the twisted region of $D_t$ for all $t$. Let $A_t$ be a transition matrix for $D_t$, $G_t$ the corresponding Green's matrix, and $\widetilde{G}_t = \text{adj}(I - A_t)$.
	
	Additionally, suppose $t \in \N$ is fixed and we are given a second index $m \in \N$ with $t > 2m > 0$. Then we define two subregions of $U_t \subset D_t$ as follows: Let $U_{t,m}^\text{ends}$ be the union of the first $m$ full twists in $U_t$ and the last $m$ full twists, and let $U_{t,m}^\text{mid} = U_t - U_{t,m}^{\text{ends}}$. In this case we refer to $m$ as the {\em ends index} of $D_t$. We also consider $D^\tau_\infty$ and $\tau_\infty$ as in \Cref{def:spec_diags}: let $A_\infty^\tau$ be a transition matrix for $D_\infty^\tau$ and let
	$$
	\widetilde{G}_\infty^\tau =  \text{Adj}(I - A_\infty^\tau).
	$$
	
	Finally, given any full twist $\tau$ in a tangle diagram, let $c_i(\tau)$ be the $i$th crossing of $\tau$ ordered from bottom to top. For any crossing $c_i(\tau_\infty)$ of $\tau_\infty \subset D^\tau_\infty$, we define $\tilde{R}_1(c_i(\tau_\infty))$ in the obvious way using $A^\tau_\infty$ and (\ref{eq:new_r_adj}). When a computation of some $\widetilde{R}_1(c_i(\tau))$ involves an infinite twist vertex, we use the power series expansion
	\begin{equation}
		\label{eq:pow}
		\frac{T^{j - 1}}{1 + T + \cdots + T^{n - 1}} = \frac{T^{j - 1}(1 - T)}{1 - T^n} = T^{j - 1}(1 - T)\sum_{k = 0}^\infty T^{nk}
	\end{equation}
	for the non-zero weights associated to the vertex in Definition \ref{def:twist_vert}. It thus makes sense to consider the coefficient of $T^r$ in $\tilde{R}_1(c_i(\tau))$, which we denote by
	$$
	\big(\tilde{R}_1(c_i(\tau))\big)_r
	$$
	as in Section \ref{sec:ab}.
	
	Before proving our main result, we have the following lemma.
	\begin{lemma}
		\label{lem:tech}
		For all $r_0 \in \Z$ there exists $m > 0$ such that for any $t > 2m$, any full twist $\tau$ in $U_{t,m}^\text{mid} \subset D_t$ and any crossing $c_k(\tau)$ in $\tau$, we have
		$$
		\big(\widetilde{R}_1(c_k(\tau))\big)_r = \big(\widetilde{R}_1(c_k(\tau_\infty))\big)_r.
		$$
		for all $r \leq r_0$.
	\end{lemma}

	Informally the lemma says that if $U_t$ contains a sufficiently high number of twists, then full twists near the middle of $U_t$ behave like $\tau_\infty$ in the computation of $\rho_1$.
	
	\begin{proof}
		Fix $r_0 \in \Z$ as in the statement of the lemma and let $r_1 > 0$ be arbitrary. As in Section \ref{sec:burau}, let $\psi(\Omega^t_n)_{ij}$ denote the probability of entering a sequence of $t$ full twists on $n$ strands at the $i$th incoming strand and exiting at the $j$th outgoing one. By Proposition \ref{thm:twist_stab}, The sequence $\{\psi(\Omega^t_n)_{ij}\}_t$ stabilizes positively to the series
		$$
		\lim_{t \to \infty} \psi(\Omega^t_n)_{ij} = T^{j - 1}(1 - T)\sum_{k = 0}^\infty T^{nk}
		$$
		for all $i$ and $j$. Thus we may choose $m \in \N$ such that for all $M \geq m$, all $i$ and $j$ and all $r \leq r_1$,
		$$
		(\psi(\Omega^M_n)_{ij})_r = (T^{j - 1}(1 - T)\sum_{k = 0}^\infty T^{nk})_r.
		$$
		
		We now fix a diagram $D_t$ with $t > 2m$. Let $\tau$ be an arbitrary full twist contained in $U^\text{mid}_{t,m}$, and let $D^\tau_t$ be the Markov chain in which we contract every twist in $U_t$ other than $\tau$:
		$$
		D^\tau_t = D_t/(U_t - \tau).
		$$
		The non-contracted part of $D_t$ in the above chain consists of $D_t - U_t$ and the lone twist region $\tau$. We thus have a natural identification of states  $\mathcal{S}(D^\tau_t) \leftrightarrow \mathcal{S}(D^\tau_\infty)$ given by identifying $D_t - U_t$ with $D_0 - U_0$ and $\tau$ with $\tau_\infty$. We fix the same ordering for both sets of states, and let $A^\tau_t$ be the resulting transition matrix for $D^\tau_t$.
		
		If $s_i$ and $s_j$ are two states in $\mathcal{S}(D^\tau_t)$ which are not both adjacent to a contracted region, then
		$$
		(A^\tau_t)_{ij} = (A^\tau_\infty)_{ij}
		$$
		since the diagrams are identical there. If $s_i$ is the $i$th incoming strand of a contracted region and $s_j$ is the $j$th outgoing strand of the same region, labeling from left to right (and assuming the indices match to simplify notation), then
		$$
		(A^\tau_t)_{ij} = \psi(\Omega^M_n)_{ij},
		$$
		where $M \geq m$ is the number of full twists in the contracted region. On the other hand, by definition,
		$$
		(A^\tau_\infty)_{ij} = \psi(\Omega^\infty_n)_{ij}.
		$$
		It therefore follows from our choice of $m$ that
		\begin{equation}
			\label{eq:astab}
			((A^\tau_t)_{ij})_r = ((A^\tau_\infty)_{ij})_r
		\end{equation}
		for all fixed $i$ and $j$ and all $r \leq r_1$. We emphasize that (\ref{eq:astab}) holds for any $t > 2m$ and any choice of full twist $\tau$ in $U^\text{mid}_{t,m}$.
		
		Next, $G^\tau_t = (I - A^\tau_t)^{-1}$ and $\widetilde{G}^\tau_t = \text{adj}(I - A^\tau_t)$. We have a natural inclusion $\mathcal{S}(D^\tau_t) \hookrightarrow \mathcal{S}(D_t)$ induced by the inclusion $D_t - (U_t - \tau) \hookrightarrow D_t$. To simplify notation, we index the states of $\mathcal{S}(D_t)$ by extending our indexing of $\mathcal{S}(D^\tau_t)$ via this inclusion. By Proposition \ref{thm:cmc} and Corollary \ref{cor:g_matrix} we have
		$$
		(G^\tau_t)_{ij} = (G_t)_{ij}
		$$
		for all $i$ and $j$ indexing states of $\mathcal{S}(D^\tau_t)$. Further, since $U - \tau$ does not admit cycles, Proposition \ref{thm:cmc_det} tells us
		\begin{equation}
			\label{eq:deteq}
			\det(I - A^\tau_t) = \det(I - A_t).
		\end{equation}
		It follows that for all valid $i$ and $j$,
		\begin{equation}
			\label{eq:geq}
			(\widetilde{G}^\tau_t)_{ij} = \det(I - A^\tau_t)(G^\tau_t)_{ij} = \det(I - A_t) (G_t)_{ij} = (\widetilde{G}_t)_{ij}.
		\end{equation}
	
		Let $c_k(\tau)$ be an aribitrary crossing of $\tau$ represented by the triple $(\sigma, i, j)$. Substituting (\ref{eq:deteq}) and (\ref{eq:geq}) into (\ref{eq:new_r_adj}), we have
		\begin{align}
			\label{eq:compare}
			\widetilde{R}_1(c_k(\tau)) &= \sigma((\widetilde{G}^\tau_t)_{ji}((\widetilde{G}^\tau_t)_{j^+, j} + (\widetilde{G}^\tau_t)_{j, j^+} \\
			& \ \ \ \ \ - (\widetilde{G}^\tau_t)_{ij}) - (\widetilde{G}^\tau_t)_{ii}((\widetilde{G}^\tau_t)_{j,j^+} - \det(I - A^\tau_t)) - \det(I - A^\tau_t)^2/2). \nonumber
		\end{align}
	
		We compare the terms in the above equation with 
		\begin{align*}
		\widetilde{R}_1(c_k(\tau_\infty)) &= \sigma((\widetilde{G}^\tau_\infty)_{ji}((\widetilde{G}^\tau_\infty)_{j^+, j} + (\widetilde{G}^\tau_\infty)_{j, j^+} \\
		& \ \ \ \ \ - (\widetilde{G}^\tau_\infty)_{ij}) - (\widetilde{G}^\tau_\infty)_{ii}((\widetilde{G}^\tau_\infty)_{j,j^+} - \det(I - A^\tau_\infty)) - \det(I - A^\tau_\infty)^2/2).
		\end{align*}
	
		Following \Cref{lem:polynomial}, since $A^\tau_t$ and $A^\tau_\infty$ have the same dimension and $c_k(\tau)$ and $c_k(\tau_\infty)$ have the same crossing indices, $\widetilde{R}_1(c_k(\tau))$ and $\widetilde{R}_1(c_k(\tau_\infty))$ are given by the same polynomial function on $|\mathcal{S}(D^\tau_\infty)|$-by-$|\mathcal{S}(D^\tau_\infty)|$ matrices. In other words, there exists a polynomial function $f_{|\mathcal{S}(D^\tau_\infty)|, \sigma, i , j, j^+}$ on the entries of $|\mathcal{S}(D^\tau_\infty)|$-by-$|\mathcal{S}(D^\tau_\infty)|$ matrices such that
		$$
		f_{|\mathcal{S}(D^\tau_\infty)|, \sigma, i , j, j^+}(A^\tau_\infty) = \widetilde{R}_1(c_k(\tau_\infty))
		$$
		and
		$$
		f_{|\mathcal{S}(D^\tau_\infty)|, \sigma, i , j, j^+}(A^\tau_t) =  \widetilde{R}_1(c_k(\tau)).
		$$
		
		We observe that every nonzero element of $A^\tau_\infty$ is equal either to $T^{\pm 1}$, $1 - T^{\pm 1}$, or a power series of the form given in (\ref{eq:pow}). In particular, each Laurent series entry of $A^\tau_\infty$ has only finitely many terms where the exponent of $T$ is negative. The proof of \Cref{lem:stab_cont} thus implies that for some choice of $r_1$, any matrix $A$ satisfying
		$$
		(A_{ij})_r = ((A^\tau_\infty)_{ij})_r
		$$
		for all $i$ and $j$ and all $r \leq r_1$ also satisfies
		$$
		(f_{|\mathcal{S}(D^\tau_\infty)|, \sigma, i , j, j^+}(A))_r = (f_{|\mathcal{S}(D^\tau_\infty)|, \sigma, i , j, j^+}(A^\tau_\infty))_r
		$$
		for all $r \leq r_0$. Fixing such an $r_1$ and adjusting $m$ as necessary, we conclude from (\ref{eq:astab}) that
		$$
		\big(\widetilde{R}_1(c_k(\tau))\big)_r = (f_{|\mathcal{S}(D^\tau_\infty)|, \sigma, i , j, j^+}(A^\tau_t))_r = (f_{|\mathcal{S}(D^\tau_\infty)|, \sigma, i , j, j^+}(A^\tau_\infty))_r = \big(\widetilde{R}_1(c_k(\tau_\infty))\big)_r
		$$
		for all $r \leq r_0$. Since $\tau_\infty$ contains finitely many crossings, we can choose an $m$ such that the above holds for all crossings of $\tau_\infty$ simultaneously. This proves the result.
	\end{proof}

	\begin{thm}
		\label{prop:main}
		Define
		$$
		d_t = d_t(\{K_t\}) = T^{tn(n - 1)} \Big( T^{n(n - 1)}\rho_1(K_{t + 1}) - \rho_1(K_t) \Big) \in \Z[T,T^{-1}]
		$$
		for all $t \in \N$. Then the sequence $\{d_t\}$ stabilizes positively with limit
		$$
		\lim_{t \to \infty} d_t = T^{-\varphi(D_0) -w(D_0)} \sum_{k = 1}^{n(n - 1)} \widetilde{R}_1(c_k(\tau_\infty)),
		$$
		where $\tau_\infty$ is the distinguished full twist in $D^\tau_\infty$.
	\end{thm}

	\begin{proof}
		Given an ends index $m$, for all $t > 2m$ we write $\mathcal{S}_{t,m}^\text{mid}$ to indicate $\mathcal{S}(U^\text{mid}_{t,m})$ and $\mathcal{S}_{t,m}^{\overline{\text{mid}}}$ to denote its complement $\mathcal{S}(D_t) - \mathcal{S}^\text{mid}_{t,m}$. Similarly, we write $\mathcal{C}_{t,m}^\text{mid}$ to mean the set of crossings $\mathcal{C}(U^\text{mid}_{t,m})$, and $\mathcal{C}_{t,m}^{\overline{\text{mid}}}$ for the complementary set $\mathcal{C}(D_t) - \mathcal{C}^\text{mid}_{t,m}$.
		
		We write
		$$
		\rho_1(K_t) = \rho^\text{mid}_m(K_t) + \rho^{\overline{\text{mid}}}_m(K_t),
		$$
		where as in (\ref{eq:new_rho})
		$$
		\rho^\text{mid}_m(K_t) = T^{-\varphi(D_t) -w(D_t)} \Big( \sum_{c \in \mathcal{C}_{t,m}^\text{mid}} \tilde{R}_1(c) - \sum_{k \in \mathcal{S}_{t,m}^\text{mid}} \tilde{\varphi}_k \Big),
		$$
		and $\rho^{\overline{\text{mid}}}_m$ is defined analogously for $\mathcal{C}_{t,m}^{\overline{\text{mid}}}$ and $\mathcal{S}_{t,m}^{\overline{\text{mid}}}$. Then
		\begin{equation}
			\label{eq:setup}
			d_t = T^{tn(n - 1)} \big(T^{n(n - 1)} \rho^\text{mid}_m(K_{t + 1}) -  \rho^\text{mid}_m(K_t)\big) + T^{tn(n - 1)}\big(T^{n(n - 1)} \rho^{\overline{\text{mid}}}_m(K_{t + 1}) -  \rho^{\overline{\text{mid}}}_m(K_t) \big).
		\end{equation}
		Fix $r_0 \in \Z$; we consider the left grouping in (\ref{eq:setup}) first. We observe that all strands in $U^\text{mid}_{t,m}$ for any $t$ and $m$ have zero turning number, so by (\ref{eq:newphi})
		$$
		\rho^\text{mid}_m(K_t) = T^{-\varphi(D_t) -w(D_t)} \sum_{c \in \mathcal{C}_{t,m}^\text{mid}} \tilde{R}_1(c).
		$$
		Additionally, since $\varphi(D_t) = \varphi(D_0)$ for all $t$ and
		$$
		w(D_t) = w(D_0) + tn(n - 1),
		$$
		we have
		\begin{equation}
			\label{eq:first_grp}
			T^{tn(n - 1)} \big(T^{n(n - 1)} \rho^\text{mid}_m(K_{t + 1}) -  \rho^\text{mid}_m(K_t)\big) = T^{-\varphi(D_0) -w(D_0)}\big(\sum_{c \in \mathcal{C}_{t + 1,m}^\text{mid}} \tilde{R}_1(c) - \sum_{c \in \mathcal{C}_{t,m}^\text{mid}} \tilde{R}_1(c)\big).
		\end{equation}
		Using Lemma \ref{lem:tech} we choose an ends index $m$ such that for any $t > 2m$, any full twist $\tau$ in $U^\text{mid}_{t,m}$ and any crossing $c_k(\tau)$,
		$$
		\big(\widetilde{R}_1(c_k(\tau))\big)_r = \big(\widetilde{R}_1(c_k(\tau_\infty))\big)_r
		$$
		for all $r \leq r_0$. This implies
		$$
		\big(\sum_{k = 1}^{n(n-1)} \widetilde{R}_1(c_k(\tau))\big)_r = \big(\sum_{k = 1}^{n(n-1)} \widetilde{R}_1(c_k(\tau_\infty))\big)_r
		$$
		for all $r\leq r_0$, so in fact
		$$
		\big( \sum_{c \in \mathcal{C}_{t,m}^\text{mid}} \tilde{R}_1(c) \big)_r = \big((t - 2m) \cdot \sum_{r = 1}^{n(n-1)} \widetilde{R}_1(c_r(\tau_\infty)) \big)_r
		$$
		for all $r \leq r_0$ and all $t > 2m$. From this and (\ref{eq:first_grp}), we conclude that
		\begin{equation}
			\label{eq:plugin}
			\Big( T^{tn(n - 1)} \big(T^{n(n - 1)} \rho^\text{mid}_m(K_{t + 1}) -  \rho^\text{mid}_m(K_t)\big) \Big)_r = \Big(T^{-\varphi(D_0) -w(D_0)} \sum_{k = 1}^{n(n-1)} \widetilde{R}_1(c_k(\tau_\infty)) \Big)_r
		\end{equation}
		for all $r \leq r_0$ and all $t > 2m$.
		
		We finish the proof of the theorem by showing the second grouping in (\ref{eq:setup}) stabilizes positively to zero. For this, define
		$$
		D^{\overline{\text{mid}}}_{t,m} = D_t/U^\text{mid}_{t,m}.
		$$
		Additionally, let $D^{\overline{\text{mid}}}_{\infty,m}$ be the diagram formed by replacing the region $U^\text{mid}_{t,m}$ with a single infinite twist vertex. We note that the choice of $t$ does not matter for this definition, so long as $t > 2m$. Let $A^{\overline{\text{mid}}}_{t,m}$, $A^{\overline{\text{mid}}}_{\infty,m}$, $G^{\overline{\text{mid}}}_{t,m}$ and $G^{\overline{\text{mid}}}_{\infty,m}$ be transition and Green's matrices for $D^{\overline{\text{mid}}}_{t,m}$ and $D^{\overline{\text{mid}}}_{\infty,m}$ respectively, and let
		\begin{align*}
		\widetilde{G}^{\overline{\text{mid}}}_{t,m} &= \text{adj}(I - A^{\overline{\text{mid}}}_{t,m}) \\
		\widetilde{G}^{\overline{\text{mid}}}_{\infty,m} &= \text{adj}(I - A^{\overline{\text{mid}}}_{\infty,m}).
		\end{align*}
	
		The diagrams $D_t - U^\text{mid}_{t,m}$ are identical for all $t > 2m$, and we identify the sets of states $\mathcal{S}_{t,m}^{\overline{\text{mid}}}$ with each other and with $\mathcal{S}(D^{\overline{\text{mid}}}_{\infty,m})$ in the obvious way. We choose an indexing for this set, and extend this to an indexing on $\mathcal{S}(D_t)$ for each $t$ via the inclusion $D_t - U^\text{mid}_{t,m} \hookrightarrow D_t$. Similarly, we identify the sets of crossings $\mathcal{C}^{\overline{\text{mid}}}_{t,m}$ and $\mathcal{C}(D^{\overline{\text{mid}}}_{\infty,m})$ for all $t$.
		
		As in the proof of Lemma \ref{lem:tech}, since the contracted region $U^\text{mid}_{t,m}$ admits no cycles, Propositons \ref{thm:cmc} and \ref{thm:cmc_det} and Corollary \ref{cor:g_matrix} give
		$$
		\det(I - A_t) = \det(I - A^{\overline{\text{mid}}}_{t,m})
		$$
		and
		$$
		(\widetilde{G}_t)_{ij} = \det(I - A_t)(G_t)_{ij} = \det(I - A^{\overline{\text{mid}}}_{t,m})(G^{\overline{\text{mid}}}_{t,m})_{ij} = (\widetilde{G}^{\overline{\text{mid}}}_{t,m})_{ij}
		$$
		for all $i$ and $j$ indexing states in $\mathcal{S}_{t,m}^{\overline{\text{mid}}}$. The same arguments used in the proofs of \Cref{lem:alex} and Lemma \ref{lem:tech} also show that
		\begin{equation}
			\label{eq:final_stab}
			\lim_{t \to \infty} A^{\overline{\text{mid}}}_{t,m} = A^{\overline{\text{mid}}}_{\infty,m},
		\end{equation}
		which implies 
		$$
		\lim_{t \to \infty} \widetilde{G}^{\overline{\text{mid}}}_{t,m} = \widetilde{G}^{\overline{\text{mid}}}_{\infty,m}
		$$
		by \Cref{lem:stab_cont}.
		
		Given a fixed crossing $c$ of $D_{2m + 1} - U^\text{mid}_{2m + 1,m}$, let $c^t$ denote the corresponding crossing of $\mathcal{C}^{\overline{\text{mid}}}_{t,m}$ and $c^\infty$ the same crossing in $\mathcal{C}(D^{\overline{\text{mid}}}_{\infty,m})$. By the above discussion, for any such $c$ with triple $(\sigma, i,j)$, we have
		\begin{align*}
		\lim_{t \to \infty} \widetilde{R}_1(c^t) &= \lim_{t \to \infty}\sigma((\widetilde{G}_t)_{ji}((\widetilde{G}_t)_{j^+, j} + (\widetilde{G}_t)_{j, j^+} \\
		& \ \ \ \ \ - (\widetilde{G}_t)_{ij}) - (\widetilde{G}_t)_{ii}((\widetilde{G}_t)_{j,j^+} - \det(I - A_t)) - \det(I - A_t)^2/2) \\
		&= \lim_{t \to \infty}\sigma((\widetilde{G}^{\overline{\text{mid}}}_{t,m})_{ji}((\widetilde{G}^{\overline{\text{mid}}}_{t,m})_{j^+, j} + (\widetilde{G}^{\overline{\text{mid}}}_{t,m})_{j, j^+} \\
		& \ \ \ \ \ - (\widetilde{G}^{\overline{\text{mid}}}_{t,m})_{ij}) - (\widetilde{G}^{\overline{\text{mid}}}_{t,m})_{ii}((\widetilde{G}^{\overline{\text{mid}}}_{t,m})_{j,j^+} - \det(I - A^{\overline{\text{mid}}}_{t,m})) - \det(I - A^{\overline{\text{mid}}}_{t,m})^2/2) \\
		&= \sigma((\widetilde{G}^{\overline{\text{mid}}}_{\infty,m})_{ji}((\widetilde{G}^{\overline{\text{mid}}}_{\infty,m})_{j^+, j} + (\widetilde{G}^{\overline{\text{mid}}}_{\infty,m})_{j, j^+} \\
		& \ \ \ \ \ - (\widetilde{G}^{\overline{\text{mid}}}_{\infty,m})_{ij}) - (\widetilde{G}^{\overline{\text{mid}}}_{\infty,m})_{ii}((\widetilde{G}^{\overline{\text{mid}}}_{\infty,m})_{j,j^+} - \det(I - A^{\overline{\text{mid}}}_{\infty,m})) - \det(I - A^{\overline{\text{mid}}}_{\infty,m})^2/2) \\
		&= \tilde{R}_1(c^\infty).
		\end{align*}
		Similarly, if $s^t$ is a fixed state of $D_t - U^\text{mid}_{t,m}$, let $s^\infty$ be the corresponding state of $\mathcal{S}(D^{\overline{\text{mid}}}_{\infty,m})$. Since the turning number $\varphi(s^t)$ does not depend on the value of $t$, a computation analogous to the one above shows
		$$
		\lim_{t \to \infty} \tilde{\varphi}(s^t) = \tilde{\varphi}(s^\infty).
		$$
		Returning to the second grouping of (\ref{eq:setup}), these two computations together imply that
		\begin{align*}
		&\lim_{t \to \infty} T^{n(n - 1)} \rho^{\overline{\text{mid}}}_m(K_{t + 1}) -  \rho^{\overline{\text{mid}}}_m(K_t) \\
		& \ \ \ \ \ = \lim_{t \to \infty}  T^{-w(D_0) - \varphi(D_0)}\big(\sum_{c \in \mathcal{C}_{{t+1},m}^{\overline{\text{mid}}}} \tilde{R}_1(c) - \sum_{k \in \mathcal{S}_{{t+1},m}^{\overline{\text{mid}}}} \tilde{\varphi}(k) - \sum_{c \in \mathcal{C}_{t,m}^{\overline{\text{mid}}}} \tilde{R}_1(c) + \sum_{k \in \mathcal{S}_{t,m}^{\overline{\text{mid}}}} \tilde{\varphi}(k) \big) \\
		& \ \ \ \ \ = T^{-w(D_0) - \varphi(D_0)}\big(\sum_{c \in \mathcal{C}(D^{\overline{\text{mid}}}_{\infty,m})} \tilde{R}_1(c) - \sum_{k \in \mathcal{S}(D^{\overline{\text{mid}}}_{\infty,m})} \tilde{\varphi}(k) - \sum_{c \in \mathcal{C}(D^{\overline{\text{mid}}}_{\infty,m})} \tilde{R}_1(c) + \sum_{k \in \mathcal{S}(D^{\overline{\text{mid}}}_{\infty,m})} \tilde{\varphi}(k) \big) \\
		& \ \ \ \ \ = 0.
		\end{align*}
		Thus we can choose $t_0 > 2m$ such that for all $t > t_0$ and $r \leq r_0$.
		$$
		\big( T^{n(n - 1)} \rho^{\overline{\text{mid}}}_m(K_{t + 1}) -  \rho^{\overline{\text{mid}}}_m(K_t) \big)_r = 0.
		$$
		
		Subsituting (\ref{eq:plugin}) and the above equation into (\ref{eq:setup}), we conclude that for all $t > t_0$ and all $r \leq r_0$,
		$$
		(d_t)_r = \Big(T^{-\varphi(D_0) -w(D_0)} \sum_{k = 1}^{n(n-1)} \widetilde{R}_1(c_k(\tau_\infty)) \Big)_r.
		$$
		Since $r_0$ was arbitrary, $d_t$ stabilizes positively to the desired limit.
	\end{proof}

	It is now straightforward to compute the asymptotic growth rate of $\rho_1$ for any twisted family of knots. As in Section \ref{sec:alex}, we consider the family of torus knots $\{{\bf T}(2,2T+1)\}$. For this family the diagram $D^\tau_\infty$ looks as in Figure \ref{fig:d_tau_inf} with transition matrix given by:

	$$
	A = \begin{bmatrix}
		0 & T & 0 & 0 & 0 & 0 & 1 - T & 0 & 0 & 0 & 0 \\
		0 & 0 & \frac{T}{1 + T} & 0 & 0 & 0 & 0 & \frac{1}{1 + T} & 0 & 0 & 0 \\
		0 & 0 & 0 & 1 & 0 & 0 & 0 & 0 & 0 & 0 & 0 \\
		0 & 0 & 0 & 0 & T & 0 & 0 & 0 & 0 & 1 - T & 0 \\
		0 & 0 & 0 & 0 & 0 & \frac{T}{1 + T} & 0 & 0 & 0 & 0 & \frac{1}{1 + T} \\
		0 & 0 & 0 & 0 & 0 & 0 & 1 & 0 & 0 & 0 & 0 \\
		0 & 0 & \frac{T}{1 + T} & 0 & 0 & 0 & 0 & \frac{1}{1 + T} & 0 & 0 & 0 \\
		0 & 0 & 0 & 1 - T & 0 & 0 & 0 & 0 & T & 0 & 0 \\
		0 & 0 & 0 & 0 & 0 & 0 & 0 & 0 & 0 & 1& 0 \\
		0 & 0 & 0 & 0 & 0 & \frac{T}{1 + T} & 0 & 0 & 0 & 0 & \frac{1}{1 + T} \\
		0 & 0 & 0 & 0 & 0 & 0 & 0 & 0 & 0 & 0 & 0 
	\end{bmatrix}
	$$
	It is trivial to compute $G = (I - A)^{-1}$ and subsequently calulate the following using the formula in Theorem \ref{prop:main}.
	
	\begin{named_thm}{\refthm{example}}
		The asymptotic growth rate of the $\rho_1$ invariant for the family of $(2,q)$-torus knots is equal to
		$$
		-\frac{1}{(1 + T)^2} = -1 + 2T - 3T^2 + 4T^3 - \cdots
		$$
	\end{named_thm}

	One may also notice that the determinant of $I - A$ is the limit of Alexander polynomials $\lim_{t \to \infty} \Delta_{K_t}$ calculated in Section \ref{sec:alex}. This is always the case.

	\begin{lemma}
		\label{lem:final}
		For any coherently oriented family of twisted knots $\{K_t\}$, let $D_\infty$ and $D^\tau_\infty$ be as in \Cref{def:spec_diags} with respective transition matrices $A_\infty$ and $A^\tau_\infty$. Then
		$$
		\det(I - A^\tau_\infty) = \det(I - A_\infty) = T^\alpha \lim_{t \to \infty} \Delta_{K_t}
		$$
		where $\alpha$ is the constant of \Cref{lem:alex}. In particular, it follows from \Cref{thm:alex_limit} that $I - A^\tau_\infty$ is invertible and the Green's matrix $G^\tau_\infty$ is defined.
	\end{lemma}

	\begin{proof}
		The second equality in Lemma \ref{lem:final} is \Cref{cor:pre_lim_alex}; we must prove the first equality.
		
		Let $U \subset D^\tau_\infty$ be the union of the two infinite twist vertices and the full twist $\tau_\infty$, and let $A_{D^\tau_\infty/U}$ be a transition matrix for the contraction $D^\tau_\infty/U$. Since $U$ contains no cycles, the proof of \Cref{thm:cmc_det} gives
		$$
		\det(I - A_{D^\tau_\infty/U}) = \det(I - A^\tau_\infty)
		$$
		and it suffices to show that 
		\begin{equation}
			\label{eq:same}
			\det(I - A_{D^\tau_\infty/U}) = \det(I - A_\infty).
		\end{equation}
		In fact the two Markov chains $D^\tau_\infty/U$ and $D_\infty$ are the same. As in Section \ref{sec:burau}, let $\psi(\Omega^\infty_n)$ be the matrix such that $\psi(\Omega^\infty_n)_{ij}$ is the probability of entering an infinite twist vertex at the $i$th incoming edge and exiting at the $j$ outgoing one. Let $\psi(\Omega_n)$ be the analogous matrix (given by the Burau representation) associated to the full twist $\tau_\infty \subset U$. Using the fact that each column of $\psi(\Omega^\infty_n)$ is constant and each row of $\psi(\Omega_n)$ sums to one, it is easy to check that
		$$
		\psi(\Omega_n) \cdot \psi(\Omega^\infty_n) = \psi(\Omega^\infty_n).
		$$
		In fact
		$$
		\psi(\Omega^\infty_n) \cdot \psi(\Omega_n) \cdot \psi(\Omega^\infty_n) = \psi(\Omega^\infty_n).
		$$
		This makes intuitive sense, since concatenating some full twists with an infinite sequence of twists just yields another infinite sequence of twists. The above equation says that the probabilty of entering $U$ at the $i$th incoming strand and exiting at the $j$th outgoing one is the same as the associated probability for a single infinite twist vertex. It follows from this that if we identify the states $\mathcal{S}(D^\tau_\infty/U)$ and $\mathcal{S}(D_\infty)$ in the obvious way, then we have
		$$
		A_\infty = A_{D^\tau_\infty/U}.
		$$
		This implies equation (\ref{eq:same}).
	\end{proof}

	We use the preceeding lemma to prove a result analogous to \Cref{thm:alex_limit}. For the $\rho_1$ invariant we have the following proposition, which completes the proof of \Cref{thm:main}.
	
	\begin{prop}
		\label{thm:rho_limit}
		Let $\{K_t\}$ be a family of knots twisted along $n$ coherently oriented strands. Then the asymptotic growth rate of the $\rho_1$ invariant, viewed as a rational function, satisfies
		$$
		(\lim_{t \to \infty} d_t) |_{T = 1} = \pm \frac{n - 1}{2n}.
		$$
		In particular, $\lim_{t \to \infty} d_t$ is not a polynomial.
	\end{prop}

	\begin{proof}
		Let $D^\tau_\infty$ and $\tau_\infty \subset D^\tau_\infty$ be defined as usual for some compatible family of diagrams for the $\{K_t\}$. By Proposition \ref{prop:main}, Lemma \ref{lem:final} and Lemma \ref{thm:alex_limit} we have
		\begin{align}
			\label{eq:okay}
			(\lim_{t \to \infty} d_t) |_{T = 1} &= \pm \sum_{k = 1}^{n(n - 1)} \widetilde{R}_1(c_k(\tau_\infty))|_{T = 1} \\ \nonumber
			&= \pm \det(I - A^\tau_\infty)^2|_{T = 1} \sum_{k = 1}^{n(n - 1)} R_1(c_k(\tau_\infty))|_{T = 1} \\ \nonumber
			&= \pm \frac{1}{n^2} \sum_{k = 1}^{n(n - 1)} R_1(c_k(\tau_\infty))|_{T = 1}.
		\end{align}
		The same argument in the proof of Lemma \ref{thm:alex_limit} shows the righthand value above is unchanged by crossing changes, Reidemeister moves and reorderings of the incoming or outgoing edges of infinite twist vertices that take place away from the twist region $\tau_\infty$. It thus suffices to verify the lemma for the diagram $D$ in Figure \ref{fig:rho_lim_one}, where the labeled box is $\tau_\infty$ and the two vertices are infinite twist vertices. 
		
		Each braid component of $\tau_\infty$ passes through $2n - 2$ crossing points, so each braid component of $\tau_\infty$ contributes $2n - 1$ states to the Markov chain. We order the braid components of $\tau_\infty$ from left to right, and label the $j$th state of the $i$th braid component by $s^i_j$ as in Figure \ref{fig:rho_lim_one}. At the value $T = 1$ no jumping occurs at crossings, so a particle at state $s^i_j$, $1 \leq j < 2n - 1$, will move to state $s^i_{j + 1}$ with weight $1$. The Markov chain determined by $D$ at $T = 1$ is thus equivalent to the chain determined by the diagram $D'$ in Figure \ref{fig:rho_lim_two}, where the hollow vertices only serve to differentiate visually between the state $s^i_j$ and $s^i_{j + 1}$ for $1 \leq i \leq n$, $1 \leq j < 2n - 1$. We recall also that at $T = 1$, a particle entering an infinite twist vertex on any incoming edge will exit at any given outgoing edge with weight $1/n$. In other words, infinite twist vertices do not differentiate between different incoming or different outgoing states. For the remainder of the proof, we set $T = 1$ and suppress the evaluation $|_{T = 1}$ from our notation.
		
		\begin{figure}[t]
			\labellist
			\small \hair 2pt
			\pinlabel $\tau_\infty$ at 64 223
			\pinlabel $s^1_1$ [r] at 30 173
			\pinlabel $s^2_1$ [r] at 30 135
			\pinlabel {$s^1_{2n - 1}$} [r] at 30 277
			\pinlabel {$s^2_{2n - 1}$} [r] at 30 315
			\pinlabel {$s^n_1$} [l] at 100 173
			\pinlabel {$s^n_{2n - 1}$} [l] at 100 277
			\pinlabel $s^1_1$ [r] at 396 170
			\pinlabel $s^1_2$ [r] at 396 206
			\pinlabel {$s^1_{2n - 1}$} [r] at 405 305
			\pinlabel $s^n_1$ [l] at 490 170
			\pinlabel $s^n_2$ [l] at 490 206
			\pinlabel {$s^n_{2n - 1}$} [l] at 485 300
			\endlabellist
			\subcaptionbox{$D$ \label{fig:rho_lim_one}}{
				\includegraphics[height=7.5cm]{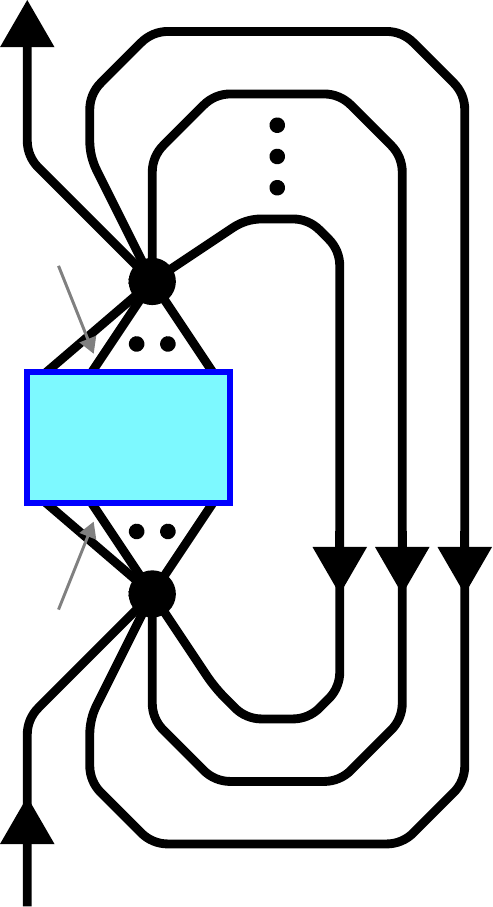}
			}
			\hspace{2cm}
			\subcaptionbox{$D'$ \label{fig:rho_lim_two}}{
				\includegraphics[height=7.5cm]{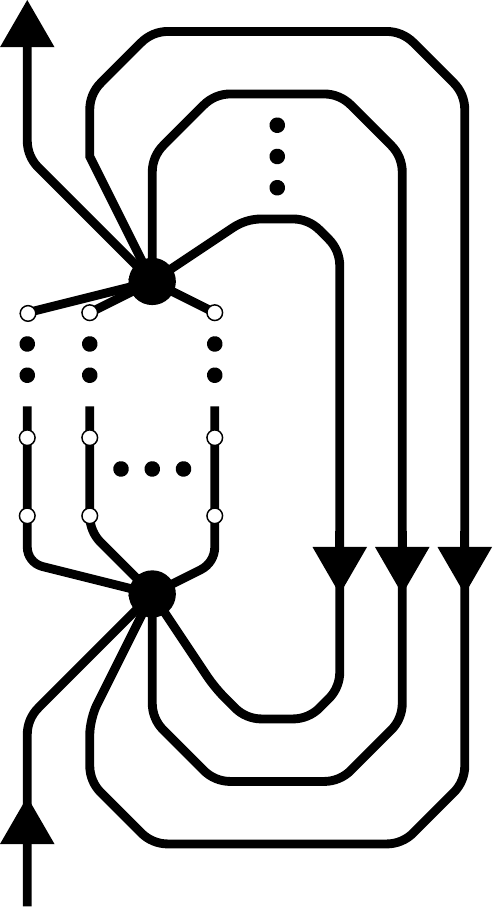}
			}
		\end{figure}
		
		Let $c$ be an arbitrary crossing of the twist region $\tau_\infty$ with incoming overstrand $s^i_j$ and incoming understrand $s^k_\ell$, $i \neq k$. Then we have
		\begin{equation}
			\label{eq:last_r}
			R_1(c) = g(s^k_\ell, s^i_j)(g(s^k_{\ell + 1}, s^k_\ell) + g(s^k_\ell, s^k_{\ell + 1}) - g(s^i_j, s^k_\ell)) - g(s^i_j,s^i_j)(g(s^k_\ell, s^k_{\ell + 1}) - 1) - 1/2.
		\end{equation}
		We consider $g$ as in (\ref{eq:prob_def}), as a weighted sum over all walks. This is valid by \Cref{lem:path_conv}, provided we can show the sum converges.
		
		Let $s^a_b$ and $s^c_d$ be states with arbitrary indices $a$, $b$, $c$ and $d$---we will compute $g(s^a_b, s^c_d)$ explicitly. Let $w$ be an arbitrary walk from $s^a_b$ to $s^c_d$. By the above discussion, the weight $a(w)$ depends only on the number of times $w$ passes through an infinite twist vertex: each such transition occurs with weight $1/n$, while any transition not involving an infinite twist vertex has weight $1$. Equivalently, let $\#_\infty(w)$ denote the number of times the walk $w$ passes through an infinite twist vertex. Then
		$$
		a(w) = \frac{1}{n^{\#_\infty(w)}}.
		$$
		Since $s^a_b$ and $s^c_d$ both lie in $\tau_\infty$, there are no walks $w \in \mathcal{W}_{s^a_b,s^c_d}$ with $\#_\infty(w)$ odd. Additionally, there exists a walk $w \in \mathcal{W}_{s^a_b,s^c_d}$ with $\#_\infty(w) = 0$ if and only if $a = c$ and $b \leq d$: in this case $w$ the unique walk given by moving along the braid component of $\tau_\infty$ from $s^a_b$ to $s^c_d$.
		
		There are $n - 1$ walks $w \in \mathcal{W}_{s^a_b,s^c_d}$ with $\#_\infty(w) = 2$: any such walk begins at $s^a_b$, moves to the infinite twist vertex and exits at one of the $n - 1$ non-outgoing strands of $D$, then moves to second twist vertex and exits at strand $c$. More generally, it is not difficult to show that
		$$
		|\{w \in \mathcal{W}_{s^a_b,s^c_d} \mid \#_\infty(w) = 2k, k > 0\}| = n^{k - 1}(n - 1)^k.
		$$
		We therefore compute
		\begin{align*}
			g(s^a_b, s^c_d) &= \sum_{w \in \mathcal{W}_{s^a_b,s^c_d}} a(w) \\
										 &= \sum_{k = 0}^\infty \sum_{\substack{w \in \mathcal{W}_{s^a_b,s^c_d} \\ \#_\infty(w) = 2k}} \frac{1}{n^{2k}} \\
										 &= \varepsilon^{a,b}_{c,d} + \sum_{k = 1}^\infty \frac{n^{k - 1}(n - 1)^k}{n^{2k}} = \varepsilon^{a,b}_{c,d} + \frac{n - 1}{n},
		\end{align*}
		where
		$$
		\varepsilon^{a,b}_{c,d}  = \begin{cases} 1 & a = c \text{ and } b \leq d \\ 0 & \text{otherwise} \end{cases}.
		$$
		In particular, the sum converges for all choices of $a$, $b$, $c$ and $d$.
		
		Let $\alpha =  \frac{n - 1}{n}$. Plugging the above calculation into (\ref{eq:last_r}), we find that
		$$
		R_1(c)|_{T = 1} = \alpha(\alpha + \alpha + 1 - \alpha) - (\alpha + 1)(\alpha + 1 - 1) - 1/2= -1/2
		$$
		for all crossings $c \in \tau_\infty$. We then use (\ref{eq:okay}) to conclude
		\begin{align*}
		\Big( \lim_{t \to \infty} d_t(\{K_t\}) \Big)|_{T = 1} = \pm \frac{1}{n^2} \Big( \sum_{c \in \tau_\infty} R_1(c)|_{T = 1} \Big) = \pm \frac{1}{n^2} \cdot \frac{n(n - 1)}{2} = \pm \frac{n - 1}{2n}.
		\end{align*}
	\end{proof}
	
\section{Proof of Lemma \protect\ref{lem:deferred}}
	\label{sec:lemma}
	
	In this section we prove Lemma \ref{lem:deferred} from the proof of Proposition \ref{thm:cmc_det}. For this it will be useful to translate our Markov chain framework into the language of graph theory. From a Markov chain $M$ with transition function $a$, we construct a directed, weighted graph $\Gamma = \Gamma(M)$ as follows: the vertex set of $\Gamma$ is $V = \mathcal{S}(M)$, and there is a directed edge $e_{st}$ from a state $s$ to a state $t$ if and only if $a(s,t) \neq 0$. We let $E$ denote the edge set of $\Gamma$ and define a weight function $a : E \to \Z(T)$ by $a(e_{st}) = a(s,t)$.
	
	In this context (simple) cycles of $M$ are (simple) cycles of $\Gamma$, and the weight of a cycle $c = [e_1 \cdots e_k]$, $e_i \in E$ is given by
	$$
	a_\text{circ}(c) = \prod_{i = 1}^k a(e_i).
	$$
	As above, the square brackets defining $c$ indicate that the sequence of edges is considered only up to cyclic permutation. Note also that we allow non-simple cycles to have repeated edges.
	
	Let $\Z^{(E)}$ be the free abelian group formally generated by the edge set $E$. Given any multicycle $q = \{c_1, \dots, c_k\}$ where $c_i = [e^i_1e^i_2 \cdots e^i_{\ell_i}]$, we define an element $f(q) \in \Z^{(E)}$ by
	$$
	f(q) = \sum_{i = 1}^k \sum_{j = 1}^{\ell_i} e^i_j.
	$$
	Additionally, for a multicycle $q$ as above, we define the {\em bad set} $\text{bad}(q)$ of $q$ to be the subgraph of $\Gamma$ consisting of all edges which occur more than once in $q$ (either in different cycles or in the same cycle), and all vertices which are the initial vertex of more than one edge in $q$. Simple multicycles are precisely those multicycles with empty bad set.
	
	\begin{lemma}
		\label{lem:annoying1}
		For any two multicycles $q$,$q'$, if $f(q) = f(q') \in \Z^{(E)}$ then
		$$
		a_\text{circ}(q) = a_\text{circ}(q')
		$$
		and bad$(q) = \text{bad}(q')$.
	\end{lemma}
	
	\begin{proof}
		It is straightforward to check that, if
		$$
		f(q) = \sum_i b_ie_i
		$$
		with $b_i \in \Z$ and $e_i \in E$, then
		$$
		a_\text{circ}(q) = \prod_i a(e_i)^{b_i}.
		$$
		Thus $a_\text{circ}(q)$ is determined by $f(q)$. Similarly, an edge of $\Gamma$ is in bad$(q)$ if and only if it appears in $f(q)$ with coefficient greater than one, and a vertex of $\Gamma$ is in bad$(q)$ if and only if it borders a bad edge or is the initial vertex of more than one edge in $f(q)$.
	\end{proof}

	For any $w \in \Z^{(E)}$, we denote by $f^{-1}(w)$ the (possible empty) set of all multicycles $q$ such that $f(q) = w$.
	
	\begin{lemma}
		\label{lem:fun1}
		Suppose $w \in \Z^{(E)}$ satisfies $w = f(q)$ for some multicycle $q$, such that the graph bad$(q) \subset \Gamma$ is non-empty and contains no cycles. Then $f^{-1}(w)$ has even cardinality, and exactly half the multicycles in $f^{-1}(w)$ have an even number of cycles.
	\end{lemma}
	
	\begin{proof}
		By Lemma \ref{lem:annoying1} every multicycle in $f^{-1}(w)$ has the same bad set, and by hypothesis bad$(q) \subset \Gamma$ is a finite, non-empty, directed acyclic graph. Thus bad$(q)$ contains at least one vertex $v$ such that no outgoing edge of $v$ in $\Gamma$ is contained in $\text{bad}(q)$. Let $\{e_1, \dots, e_n\} \subset E$ be the set of edges appearing in $w$ with nonzero coefficient which have $v$ as their initial vertex. Since none of these edges is in bad$(q)$, each appears in $w$ with coefficient one. Additionally, because $w$ can be written as a sum of cycles, $n$ is also the sum of the coefficients of all edges in $w$ which have $v$ as their {\em terminal} vertex---we can think of $n$ as the number of times the vertex $v$ is passed through when all the cycles in $q$ are traversed. Because $v \in \text{bad}(q)$, $n \geq 2$.
		
		Let $q$ be an arbitary multicycle in $f^{-1}(w)$; then if all the cycles in $q$ are traversed, the vertex $v$ is encountered $n$ times. We enumerate these ``encounters'' arbitrarily by $\varepsilon_1, \dots, \varepsilon_n$, and assume without loss of generality that on encounter $\varepsilon_i$ the multicycle $q$ exits the vertex $v$ via the edge $e_i$, for $i = 1, \dots, n$. Let $\sigma$ be an element of the permutation group $S_n$ on $n$ elements. Then we define a new multicycle $q_\sigma \in f^{-1}(w)$ by letting $q_\sigma$ be identical to $q$ except that on encounter $\varepsilon_i$ with the vertex $v$, $q_\sigma$ exits $v$ via the edge $e_{\sigma(i)}$ instead of the edge $e_i$. If we think of $q$ as a set of immersed loops in $\Gamma$, we can also conceptualize this operation as cutting each loop at the vertex $v$ and regluing the loose ends according to the permutation $\sigma$, as in Figure \ref{fig:transpose}. Since the edges $e_1, \dots, e_n$ are all distinct, $q_\sigma = q$ if and only if $\sigma$ is the identity permutation. By the same logic, given two permutations $\sigma$ and $\sigma'$, $q_\sigma = q_{\sigma'}$ if and only if $\sigma = \sigma'$.
		
		We define a partition on $f^{-1}(w)$ by letting $q \sim q'$ for $q, q' \in f^{-1}(w)$ if $q' = q_\sigma$ for some $\sigma \in S_n$. It is easy to check that this is a well-defined equivalence relation, and by the preceeding discussion each equivalence class contains exactly $|S_n| = n!$ elements. Thus $|f^{-1}(w)|$ is divisible by $n!$ for some $n \geq 2$, so $|f^{-1}(w)|$ is even. Finally we observe that if $q \in f^{-1}(w)$ and $\sigma \in S_n$ is a transposition, then $q_\sigma$ has exactly one more or one fewer cycle than $q$. Topologically, this is the situation shown in Figure \ref{fig:transpose}. It follows that $|q_\sigma|$ has the same parity as $|q|$ if and only if $\sigma$ is in the alternating group $A_n < S_n$. Since $A_n$ is an index two subgroup of $S_n$, exactly half of the multicycles in each equivalence class of $f^{-1}(w)$ contain an even number of cycles. This proves the lemma.
	\end{proof}

	\begin{figure}[t]
		\includegraphics[height=1.5cm]{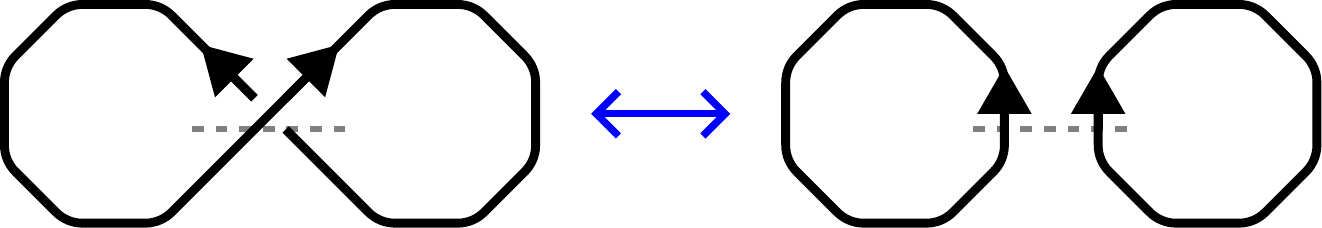}
		\caption{Changing a multicycle by a transposition}
		\label{fig:transpose}
	\end{figure}

	We now recall Lemma \ref{lem:deferred}, reframing it in our graph-theoretic context. Let $U$ be a subgraph of $\Gamma$ which admits no cycles, and let $\mathcal{Q}'_\text{bad}$ be the set of all multicycles $q$ on $\Gamma$ such that bad$(q)$ is non-empty and contained in $U$. Then the following easily implies Lemma \ref{lem:deferred}.
	
	\begin{lemma}
	$$
		\sum_{q \in \mathcal{Q}_\text{bad}'} (-1)^{|q|}a_\text{circ}(q) = 0.
	$$
	\end{lemma}

	\begin{proof}
		Fix $q' \in \mathcal{Q}_\text{bad}'$, and let $w = f(q') \in \Z^{(E)}$. Then by Lemma \ref{lem:annoying1}, since any $q \in f^{-1}(w)$ has the same bad set as $q'$,
		$$
		f^{-1}(w) \subset  \mathcal{Q}_\text{bad}'.
		$$
		We can thus partition $\mathcal{Q}_\text{bad}'$ as
		$$
		\mathcal{Q}_\text{bad}' = f^{-1}(w_1) \sqcup f^{-1}(w_2) \sqcup \cdots \sqcup f^{-1}(w_m)
		$$
		for some appropriate choice of vectors $w_1, \dots, w_m \in \Z^{(E)}$, and it suffices to show that
		$$
		\sum_{q \in f^{-1}(w)} (-1)^{|q|}a_\text{circ}(q) = 0.
		$$
		for our original arbitrary choice of $w$. Lemma \ref{lem:annoying1} tells us any $q \in f^{-1}(w)$ satisfies $a_{\text{circ}}(q) = a_{\text{circ}}(q')$. Further, since $U$ contains no cycles, bad$(q')$ contains no cycles and Lemma \ref{lem:fun1} tells us exactly half the elements of $f^{-1}(w)$ have even cardinality. Thus
		$$
		\sum_{q \in f^{-1}(w)} (-1)^{|q|}a_\text{circ}(q) = a_\text{circ}(q')\sum_{q \in f^{-1}(w)} (-1)^{|q|} = 0
		$$
		as desired.
	\end{proof}

\end{document}